\documentclass[11pt]{amsart}

\usepackage{amssymb}
\usepackage{latexsym}
\usepackage{amsmath}
\usepackage{paralist}
\usepackage[all]{xy}
\usepackage[active]{srcltx}
\usepackage{footnote}

\newtheorem{theorem}{Theorem}[section]
\newtheorem{definition}[theorem]{Definition}
\newtheorem{lemma}[theorem]{Lemma}
\newtheorem{proposition}[theorem]{Proposition}
\newtheorem{corollary}[theorem]{Corollary}
\newtheorem{example}[theorem]{Example}
\newtheorem{remark}[theorem]{Remark}
\newtheorem*{acknow}{Acknowledgements}
\newtheorem*{question}{Question}

\long\def\symbolfootnote[#1]#2{\begingroup%
\def\thefootnote{\fnsymbol{footnote}}\footnote[#1]{#2}\endgroup}

\makeatletter
\def\@cite#1#2{{\m@th\upshape\bfseries%
[{#1\if@tempswa{\m@th\upshape\mdseries, #2}\fi}]}} \makeatother

\newcommand{\bbA}{{\mathbb{A}}}
\newcommand{\bbC}{{\mathbb{C}}}
\newcommand{\bbD}{{\mathbb{D}}}
\newcommand{\bbN}{{\mathbb{N}}}
\newcommand{\bbQ}{{\mathbb{Q}}}
\newcommand{\bbR}{{\mathbb{R}}}
\newcommand{\bbT}{{\mathbb{T}}}
\newcommand{\bbZ}{{\mathbb{Z}}}

\newcommand{\A}{{\mathcal{A}}}
\newcommand{\B}{{\mathcal{B}}}
\newcommand{\C}{{\mathcal{C}}}
\newcommand{\E}{{\mathcal{E}}}
\newcommand{\F}{{\mathcal{F}}}
\newcommand{\I}{{\mathcal{I}}}
\newcommand{\J}{{\mathcal{J}}}
\renewcommand{\L}{{\mathcal{L}}}
\newcommand{\M}{{\mathcal{M}}}
\newcommand{\N}{{\mathcal{N}}}
\newcommand{\R}{{\mathcal{R}}}
\newcommand{\T}{{\mathcal{T}}}
\newcommand{\U}{{\mathcal{U}}}

\newcommand{\fA}{{\mathfrak{A}}}

\newcommand{\fC}{{\mathfrak{C}}}
\newcommand{\fD}{{\mathfrak{D}}}

\def\ga{\alpha}

\def\om{\omega}
\def\gd{\delta}
\def\gi{\iota}
\def\gm{\gamma}
\def\eps{\varepsilon}

\newcommand{\sca}[1]{\left\langle#1\right\rangle} 
\newcommand{\nor}[1]{\left\Vert #1\right\Vert}    
\newcommand\diag{\mathop{\rm diag}}

\newcommand\Span{\mathop{\rm span}}
\newcommand{\ca}{\mathrm{C}^*}
\newcommand\id{\mathop{\rm id}}
\newcommand\oC{\mathrm{C}}

\newcommand\0{\vec{\mathtt{0}}}
\newcommand\tx{\vec{\mathtt{x}}}
\newcommand\ty{\vec{\mathtt{y}}}
\newcommand\tX{\widetilde{X}}
\newcommand\tphi{\tilde{\phi}}

\begin{document}

\title[Semicrossed products of $\ca$-algebras]
{Semicrossed products of $\ca$-algebras\\ and their $\ca$-envelopes}

\author{Evgenios~T.A.~Kakariadis}
\address{Pure Math.\ Dept.\\U. Waterloo\\Waterloo, ON\;
N2L--3G1\\CANADA}
\email{ekakaria@uwaterloo.ca}

\thanks{The author was supported by the ``Herakleitus II'' program
(co-financed by the European Social Fund, the European Union and
Greece).}

\maketitle

\begin{abstract}
Let $\C$ be a $\ca$-algebra and $\ga\colon\C\rightarrow \C$ a unital $*$- endomorphism. There is a natural way to construct operator algebras which are called semicrossed products, using a convolution induced by the action of $\ga$ on $\C$. We show that the $\ca$-envelope of a semicrossed product is (a full corner of) a crossed product. As a consequence, we get that, when $\ga$ is $*$-injective, the semicrossed products are completely isometrically isomorphic and share the same $\ca$-envelope, the crossed product $\C_\infty \rtimes_{\ga_\infty} \bbZ$.

We show that minimality of the dynamical system $(\C,\ga)$ is equivalent to non-existence of non-trivial Fourier invariant ideals in the $\ca$-envelope. We get sharper results for commutative dynamical systems.
\end{abstract}

\section{Introduction and preliminaries}

The purpose of this paper is to give a clear picture of how non-selfadjoint operator algebras arise by a dynamical system consisting of a $\ca$-algebra and a $*$-homomorphism, i.e. a semicrossed product. Our examination focusses on finding the $\ca$-envelope of the semicrossed products subject to \emph{a covariance relation}, for which we give a full answer in every possible case. We see the $\ca$-envelope as the appropriate candidate for a $\ca$-algebra that inherits some of the properties of the dynamical system has and we show how this is justified for commutative systems. For completeness we have included in Section \ref{left} an application of the joint work \cite{KakKat11} with Elias Katsoulis. We prove some of the results of \cite{KakKat11} needed here, in an ad-hoc manner, avoiding using the language of $\ca$-correspondences, hence preserving self-containment of the paper. (The careful reader will notice though that the current paper and \cite{KakKat11} were conducted in parallel.)

Given a dynamical system $(\C,\ga)$ there are various ways of considering universal $\ca$-algebras over collections of pairs $(\pi,V)$, so that $(H,\pi)$ is a representation of $\C$, $V$ an operator in $\B(H)$ and ``a covariance relation'' holds. Some of the forms the covariance relation may have are
\begin{align*}
&(1)\quad \pi(\ga(c))=V^*\pi(c)V, &\text{ (``implements'') }\\
&(2)\quad V\pi(\ga(c))=\pi(c)V, &\text{ (``intertwines'') }\\
&(3)\quad V\pi(\ga(c))V^*=\pi(c), &\text{ (``undoes'') }
\end{align*}
and some possible choices for $V$ is to be a contraction, an isometry, a co-isometry or a unitary\symbolfootnote[1]{\ The presentation we give here follows the list presented by M. Lamoureux in his talk in GPOTS (1999), and I thank A. Katavolos for bringing this to my attention.}. For example, when $\ga$ is a $*$-automorphism and $V$ is considered a unitary, then all three relations are equivalent and the universal $\ca$-algebra is nothing else but the usual crossed product $\C \rtimes_\ga \bbZ$. Also, when $\C$ is unital, $\ga$ is injective and we consider relation (1) for co-isometries $V$, then we get the crossed product by an endomorphism of Stacey \cite{Sta93} (we assume that $\ga$ is non-unital in that case, otherwise relation (1) would force $\pi$ to be chosen non-degenerate, hence $V$ to be unitary). Moreover, it seems that some relations are absurd; for example relation (1) cannot hold for arbitrary contractions, since this would imply that
\[
V \pi(x^*)V^* V \pi(x) V^*
 =
\pi(\ga(x^*))\pi(\ga(x))
 =
\pi(\ga(x^*x))
 =
V \pi(x^*x) V^*,
\]
for all $x\in \C$, which triggers more assumptions on $V$. Indeed, for $x=e_\C$, i.e. the identity of $\C$, we obtain that $V\pi(e_\C)V^*$ is a projection that commutes with $\pi(\C)$ (note that we do not know a priori that $\pi$ is non-degenerate).

It is clear that when $V$ is a co-isometry then $[(1) \Rightarrow (2) \Rightarrow (3)]$, whereas when $V$ is an isometry the reverse implications hold. Hence, we can say that relation (2) is the linking condition; for this reason we choose relation (2) to play the role of the covariance relation throughout this paper. Going even further we examine the ``adjoint'' of this relation, i.e.
\begin{align*}
&(2)^\# \quad \pi(\ga(c))V=V\pi(c), &\text{ (``intertwines'') }.
\end{align*}

\emph{A} semicrossed product is a universal non-selfadjoint operator algebra relative to covariant pairs $(\pi,V)$ that satisfy $(2)$ or $(2)^\#$ (see Definition \ref{main def}), and has been under investigation by various authors \cite{AlaP, Arv, ArvJ, DeAPet, DKconj, HadH, MM}, starting with the work of Arveson \cite{Arv} in the late sixties. They are a rather powerful tool used for the investigation of the dynamical system as they provide a complete invariant for outer conjugacy in various cases \cite{DavKak12, DKconj, DavKat08-2}. Many times, the authors make assumptions such as, commutativity of $\C$, or injectivity of $\ga$. In the same direction with \cite{KakKat10}, the main objective of this paper is to identify the $\ca$-envelope of the semicrossed products (i.e. the smaller $\ca$-cover) in the general case. Towards this we prove that the $\ca$-envelope of any semicrossed product is (a full corner of) a crossed product. Moreover, we show that when $\ga$ is injective, then the $\ca$-envelope is the usual crossed product that arises from the direct limit dynamical system $(\C_\infty,\ga_\infty)$ of Stacey \cite{Sta93}. These extend the results of Peters \cite{Petarx}, whose work was our initial point.

Likewise to the multiple choices for a semicrossed product, there is a variety of C*-algebra crossed products by endomorphisms introduced as possible generalizations of the crossed product, initiated by the work of Cuntz \cite{Cun77} (for example see \cite{BroRaeVit09, Exel03, KakKat11, Lac00, Lar10, Mur02, Pas, Sta93}). Notable examples are Stacey's crossed product $\C\rtimes^1_\ga \bbN$ \cite[Definition 3.1]{Sta93}, Exel's crossed product $\C\rtimes_{\ga,\L} \bbN$ \cite[Definition 3.7]{Exel03} and $\U(\C,\ga)$ \cite[Definitions 4.4]{Exel03}. Moreover, a connection between Stacey's and Exel's crossed products was established by an Huef and Raeburn \cite{HueRae11, HueRae11-2}, after the present paper appeared in public as a preprint (see also \cite{KakPet13}). Unlikely to \cite{Exel03} we do not assume the existence of a transfer operator, as we wanted to keep low on the assumptions that determine the classes of the representations we examine. Moreover, in many of the above cases the dynamical systems are non-unital, which case is not examined here. Nevertheless, in our unital case there is indeed a connection between $\C\rtimes_\ga^1 \bbN$, $\U(\C,\ga)$ and the $\ca$-envelope of a semicrossed product, as shown in Remark \ref{rem exe}. In addition, we remark that generalized crossed products appear as enveloping $\ca$-algebras relative to $\ca$-representations, whereas the first objective of this paper is to identify the $\ca$-envelope of enveloping non-selfadjoint operator algebras relative to contractive representations of non-involutive Banach algebras (see also the remarks at the end of this Section).

Adding to the literature of the generalized $\ca$-crossed product, we propose the $\ca$-envelope as one more possible candidate of a $\ca$-algebra connected to a dynamical system. There is a great number of results showing that the usual crossed product captures some of the properties of an automorphic dynamical system. Since the $\ca$-envelope is (a full corner of) a crossed product it seems reasonable to try and ``pull back'' these results to non-surjective dynamical systems. Towards this, we obtain results concerning minimality. In particular, for the commutative case, we prove that minimality of the $\ca$-envelope is equivalent to the compact Hausdorff space being infinite and the dynamical system being minimal. As a consequence, one gets that all the ideals in any $\ca$-cover of the semicrossed product are boundary with respect to the semicrossed product.

\subsection*{Structure of the paper}

In section 2 we give the main definitions for the operator algebras that we call \emph{the semicrossed products of $(\C,\ga)$}, by using left (resp. right) covariant pairs. Also, we develop an argument of duality that enables us to examine just the left or the right case. Thus can drop down to four operator algebras $\fA(\C,\ga,\text{contr})_l$, $\fA(\C,\ga,\text{is})_l$, $\fA(\C,\ga,\text{co-is})_l$ and $\fA(\C,\ga,\text{un})_l$.

In Section \ref{right} we show that $\fA(\C,\ga,\text{co-is})_l$ and $\fA(\C,\ga,\text{un})_l$ are completely isometrically isomorphic and that their $\ca$-envelope is a crossed product. We mention that the semicrossed product $\fA(\C,\ga,\text{is})_l$ is Peters' semicrossed product introduced in \cite{Pet84}. In that paper the author was asking about the case of the right covariance relation; this is exactly the semicrossed product $\fA(\C,\ga,\text{is})_r$ that is explored here. Theorem \ref{theorem C^*-envelope right isometric} generalizes \cite[Theorem 4]{Petarx}.

In Section \ref{left} we examine $\fA(\C,\ga,\text{contr})_l$ and $\fA(\C,\ga,\text{is})_l$ which are completelly isometrically isomorphic. In a joint work with Elias Katsoulis \cite{KakKat11} we have shown that the $\ca$-envelope of $\fA(\C,\ga,\text{is})_l$ is a full corner of a crossed product, using the language of $\ca$-correspondences. Here, we give the proof by using ad-hoc versions of the arguments of \cite{KakKat11} that preserves the self-containment of this paper.

In Section \ref{overview} we present some remarks that arise naturally from the work of the previous two sections. For example we show that the semicrossed products we construct with respect to left covariant pairs are completely isometrically isomorphic (in a natural way) if and only if the $*$-endomorphism $\ga$ is injective.

In Section \ref{minimality} we show that the $\ca$-envelope has no \emph{Fourier-invariant ideals} (see Definition \ref{Fourier defn}) if and only if the dynamical system is \emph{minimal} (see Definition \ref{defn minimal}). In this case, all four semicrossed products are completely isometrically isomorphic and their $\ca$-envelope is the crossed product $\C_\infty \rtimes_{\ga_\infty}\bbZ$. Moreover, when $\C$ is a commutative $\ca$-algebra $\oC(X)$ over a compact Haudorff space $X$, the $\ca$-envelope is simple if, and only if, the dynamical system is minimal and $X$ is infinite.

Finally we mention that all the dynamical systems $(\C,\ga)$ are considered unital, meaning that $\C$ has a unit $e$ and $\ga(e)=e$. Therefore, the operator algebras examined here are always unital. Analogous results may be obtained (with similar methods) for non-unital cases, when the homomorphism $\ga$ is non-degenerate.

\subsection*{Notation}

In what follows $\bbZ_+=\{0,1,2,\dots\}$ and we use the symbol $s$ (resp. $u$) for the unilateral (resp. bilateral) shift on $\ell^2(\bbZ_+)$ (resp. $\ell^2(\bbZ)$), given by $s(e_n)=e_{n+1}$ (resp. $u(e_n)=e_{n+1}$). Before we begin, we present some basic constructions that are associated to a pair $(\C,\ga)$.

For such a pair we can define the \emph{radical} ideal $\R_\ga =\overline{\cup_n \ker\ga^n}$; it is immediate that an element $c \in \C$ is in $\R_\ga$ if, and only if, $\lim_n\ga^n(c)=0$. Thus $\ga(\R_\ga)\subseteq \R_\ga$, $\ga^{-1}(\R_\ga)=\R_\ga$ and an injective $*$-homomorphism is defined by
\begin{align*}
\dot{\ga}\colon \C/\R_\ga \rightarrow \C/\R_\ga: c+\R_\ga \mapsto \ga(c)+\R_\ga.
\end{align*}
Note that $\R_\ga=(0)$ if, and only if, $\ga$ is injective.

There is a well known direct limit process introduced by Stacey \cite{Sta93} that associates an automorphic system $(\C_\infty,\ga_\infty)$ to any pair $(\C,\ga)$, defined by
\begin{align*}
\xymatrix{
  \C \ar[r]^{\ga} \ar[d]^\ga &
  \C \ar[r]^{\ga} \ar[d]^\ga &
  \C \ar[r]^{\ga} \ar[d]^\ga &
  \cdots \ar[r] &
  \C_\infty \ar[d]^{\ga_\infty} \\
  \C \ar[r]^\ga &
  \C \ar[r]^\ga &
  \C \ar[r]^\ga &
  \cdots \ar[r] &
  \C_\infty
}
\end{align*}
It is easy to see that $\ga_\infty$ is \emph{always} a $*$-automorphism. However $\C$ is not always embedded in $\C_\infty$, thus $\ga_\infty$ does not always extend $\ga$. Indeed, if $\gi_n\colon \C_n \rightarrow \C_\infty$
are the induced $*$-homomorphisms, then $\ker\gi_n=\R_\ga$; thus $\C/\R_\ga$ embeds in $\C_\infty$ and $\ga_\infty$ extends $\dot{\ga}$.

Note that the direct limit $\C_\infty$ may be the trivial $\ca$-algebra $\bbC$ (for example let $(\C,\ga)$ be as in Remark \ref{trivial} below).
Nevertheless, when $\ga\colon \C \rightarrow \C$ is injective, then $\R_\ga=(0)$ and the pair $(\C_\infty,\ga_\infty)$ is indeed an extension of $(\C,\ga)$.

\begin{remark}\label{trivial}
\textup{There are cases where the ideal $\R_\ga$ may be ``too large'', in the sense that $\C/ \R_\ga \simeq \bbC$. For example, let $X=\bbR^+ \cup \{+\infty\}$ be the one-point compactification of $\bbR^+$ and let $\C=\oC(X)$. Then $\C$ is the unitization of $C_0(\bbR^+)$. Define the map $\phi\colon X\rightarrow X$, by $\phi(x)=x+1$ and $\phi(\infty)=\infty$ and consider the $*$-homomorphism $\ga\colon \C \rightarrow \C$ given by $\ga(f)=f\circ \phi$.}
\end{remark}

We will also need the construction of the enveloping operator algebra of a unital Banach algebra (see \cite[2.4.6 and 2.4.7]{BleLeM04}). In a few words, let $\B$ be a unital Banach algebra and let $\F$ be a collection of (possibly degenerate) contractive representations $(H_\pi,\pi)$ of $\B$, where $H_\pi$ is a Hilbert space with dimension less or equal to an ordinal $\beta$, such that $\beta^{\aleph_0} = \beta$. For any integer $\nu \geq 1$ and $[F_{ij}] \in \M_\nu(\B)$, let the seminorms
\begin{align*}
\om_\nu([F_{ij}])
 =
\sup \left\{ \nor{[\pi(F_{ij})]}_{\M_\nu(B(H_\pi))}: (H_\pi,\pi)\in \F \right\}.
\end{align*}
If $\N=\ker \om_1$, then we can define the induced $\nu$-norms on the quotient $\B/\N$; we let $\nor{F+ \N}_\infty:= \om_1(F)$. The \emph{enveloping operator algebra $\fA(\B,\F)$ of $\B$ with respect to the collection $\F$}, is the completion of the quotient $\B/\N$ with respect to the norm $\nor{\cdot}_\infty$, and has the following (universal) property: there is a unital completely contractive homomorphism $\gi\colon \B \rightarrow \fA(\B,\F)$, whose range is dense, such that for any contractive representation $(H_\pi,\pi) \in \F$ there exists a (necessarily unique) completely contractive homomorphism $\widetilde{\pi}\colon \fA(\B,\F) \rightarrow \B(H_\pi)$ such that $\widetilde{\pi}\circ \gi= \pi$ (for details see \cite[Section 1.2]{KakPet13}).

In order to construct the operator algebras that we will call \emph{semicrossed products of $(\C,\ga)$}, we first define the following $\ell^1$-Banach algebras. Note that this is the (non-selfadjoint) analogue of the procedure that produces the usual $\ca$-crossed product. First we equip the linear space $c_{00}(\bbZ_+)\odot \C$ (the algebraic tensor product of linear spaces) with the left multiplication
\begin{align*}
(\gd_n\otimes c) \ast_l (\gd_m\otimes y)
 =
\gd_{n+m}\otimes (a^m(c)\cdot y),
\end{align*}
and we denote by $\ell^1(\bbZ_+,\C,\ga)_l$ the Banach algebra that is obtained by completing with respect to the $|\cdot|_1$-norm
\begin{align*}
\left|\sum_{n=0}^{k}\gd_n\otimes c_n\right|_1
 =
\sum_{n=0}^{k} \|c_n\|_\C.
\end{align*}
In an analogous way, we equip the linear space $\C \odot c_{00}(\bbZ_+)$ with the right multiplication
\begin{align*}
(c\otimes \gd_n) \ast_r (y\otimes \gd_m) = (c\cdot a^n(y)) \otimes
 \gd_{n+m},
 \end{align*}
and we denote by $\ell^1(\bbZ_+,\C,\ga)_r$ the Banach algebra obtained.

Note that $\ell^1(\bbZ_+,\C,\ga)_l$ and $\ell^1(\bbZ_+,\C,\ga)_r$ are isometrically isomorphic as Banach spaces but not as Banach algebras. Also, if $e$ is the unit of $\C$ then $\gd_0\otimes e$ is the unit for both algebras.

As we will see in Definition \ref{main def}, the semicrossed products of a pair $(\C,\ga)$ are the enveloping operator algebras of the $\ell^1$-Banach algebras $\ell^1(\bbZ_+,\C,\ga)_l$ and $\ell^1(\bbZ_+,\C,\ga)_r$, with respect to various collections of contractive representations.

We use the following notation concerning crossed products (see, for example, \cite{Wil07}). Let $\ga\colon \C \rightarrow \C$ be a $*$-isomorphism of the $\ca$-algebra $\C$ and fix a faithful representation $(H_0,\pi)$ of $\C$. For the Hilbert space $H=H_0\otimes\ell^2(\bbZ)$, we define $\widehat{\pi}\colon \C\rightarrow \B(H)$, so that $\widehat{\pi}(c)= \diag\{\pi(\ga^n(c)): n\in\bbZ\}$ and $U=1_{H_0}\otimes u$, where $u$ is the forward bilateral shift on $\ell^2(\bbZ)$, such that $\widehat{\pi}(c)U = U\widehat{\pi}(\ga(c))$, for all $c\in \C$. The representation $(U\times \widehat{\pi})$ of the crossed product $\C\rtimes_\ga \bbZ$ that integrates the pair $(\widehat{\pi},U)$ is called \emph{the left regular representation}. Analogously, the representation $(\widehat{\pi} \times U^*)$ that integrates the pair $(\widehat{\pi},U^*)$ is called \emph{the right regular representation}. It is known that the $\ca$-algebras that are generated by the ranges of $(U\times \widehat{\pi})$ and $(\widehat{\pi}\times U^*)$ are both $*$-isomorphic to the crossed product $\C\rtimes_\ga \bbZ$.

Finally, recall that a $\ca$-algebra $\fC$ is said to be a \textit{$\ca$-cover} for an operator algebra $\fA$, provided that there is a completely isometric homomorphism $\gi\colon \fA \rightarrow \fC$ and $\gi(\fA)$ generates $\fC$ as a $\ca$-algebra, i.e. $\fC = \ca(\gi(\fA))$. An ideal $\J \subseteq \fC$, with the property that the restriction of the natural projection $\fC\rightarrow \fC/\J$ on $\fA$ is a complete isometry, is called a \textit{boundary ideal}. \textit{The \v{S}ilov ideal $\J_\fA$ of $\fA$ in $\fC$} is the largest boundary ideal and \textit{the $\ca$-envelope $\ca_{\text{env}}(\fA)$ of $\fA$} is then the quotient $\fC/\J_\fA$. An equivalent way to define the $\ca$-envelope of an operator algebra is through the following universal property: for any $\ca$-cover $(\C,\gi)$ of $\fA$ there is a $*$-epimorphism $\Phi\colon \C \rightarrow \ca_{env}(\fA)$, such that $\Phi(\gi(a))=a$ for any $a\in \fA$.

The existence of the $\ca$-envelope was initiated by Arveson in late 60's. Hamana \cite{Ham} gave the first general proof and later Dritschel and McCullough \cite{DrMc} gave an independent proof. These independent proofs were later simplified by the author \cite{Kak11-2} and Arveson \cite{Ar06}, respectively.

The universal property of the $\ca$-envelope suggests that it is the smaller $\ca$-cover of $\fA$. Based on that one could ``rename'' the $\ca$-envelope as \textit{$\ca$-minimal}. Even though this would make the significant difference between envelopping $\ca$-algebras and the $\ca$-envelope more apparent, it seems quite hard at this point to make a complete change on a terminology that has been established for more than 40 years.

\section{Definitions}\label{definitions}

Before we give the main definitions we have to take a closer look at the representations of the $\ell^1$-Banach algebras. Let us consider first the case of $\ell^1(\bbZ_+,\C,\ga)_l$. Let $\rho\colon \ell^1(\bbZ_+,\C,\ga)_l\rightarrow \B(H)$ be a $|\cdot|_1$-contractive representation (denoted simply by $|\cdot|_1$-representation, from now on); then the restriction of $(H,\rho)$ to the $\ca$-algebra $\C$ defines a contractive homomorphism (thus a $*$-representation) of $\C$. Also, let $V=\rho(\gd_1\otimes e)$; then $V$ is a contraction in $\B(H)$ and the definition of the left multiplication gives
\begin{align*}
\pi(c)V
& =
\rho(\gd_0\otimes c) \rho(\gd_1\otimes e)
  =
\rho\big(\gd_1\otimes (\ga(c)\cdot e)\big) \\
& =
\rho(\gd_1\otimes e)\rho(\gd_0 \otimes \ga(c))
  =
V \pi(\ga(c)),
\end{align*}
for all $c\in\C$. Conversely, let $(H,\pi)$ be a $*$-representation of $\C$ and $V$ be a contraction in $H$ such that the following equality holds
\begin{equation}\label{eq:left covariance}
\pi(c)V = V\pi(\ga(c)),\ c\in \C.
\end{equation}
We define the map
\begin{align*}
(V\times \pi)\colon \ell^1(\bbZ_+,\C,\ga)_l \rightarrow \B(H): \sum_{n=0}^\infty \gd_n\otimes c_n \mapsto \sum_{n=0}^\infty V^n\pi(c_n).
\end{align*}
It is easy to check that $(V\times \pi)$ is a $|\cdot|_1$-representation of $\ell^1(\bbZ_+,\C,\ga)_l$.

Hence, $(H,\rho)$ is a $|\cdot|_1$-representation of $\ell^1(\bbZ_+,\C,\ga)_l$ if, and only if, $\rho=(V\times \pi)$ for a pair $(\pi,V)$, where $(H,\pi)$ is a $*$-representation of $\C$, $V$ is a contraction in $H$ and equality (\ref{eq:left covariance}) holds. Such pairs $(\pi,V)$ are called \emph{left covariant contractive, isometric, co-isometric or unitary} if $V\in \B(H)$ is a contraction, an isometry, a co-isometry or a unitary, respectively. We refer to equality (\ref{eq:left covariance}) as \emph{the left covariance relation}.

Analogously, there exists a bijection between the representations of the algebra $\ell^1(\bbZ_+,\C,\ga)_r$ and the \emph{right covariant} pairs $(\pi,V)$, i.e. pairs satisfying \emph{the right covariance relation} $V\pi(c)=\pi(\ga(c))V$, $c\in \C$. In this case we write
\begin{align*}
(\pi \times V)(c \otimes \gd_n):= \pi(c)V^n.
\end{align*}

\begin{definition}\label{main def}
\textup{Let $\C$ be a unital $\ca$-algebra and $\ga\colon \C\rightarrow \C$ a unital $*$-homomorphism. We define the following enveloping operator algebras of $\ell^1(\bbZ_+,\C,\ga)_l$,\\
$\bullet$ $\fA(\C,\ga,\text{contr})_l$: with respect to the collection of representations
\begin{align*}
\{(V\times \pi)\colon (\pi,V) \text{ is a left covariant contractive pair}\},
\end{align*}
$\bullet$ $\fA(\C,\ga,\text{is})_l$: with respect to the collection of representations
\begin{align*}
\{(V\times \pi): (\pi,V) \text{ is a left covariant isometric pair}\},
\end{align*}
$\bullet$ $\fA(\C,\ga,\text{co-is})_l$: with respect to the collection of representations
\begin{align*}
\{(V\times \pi): (\pi,V) \text{ is a left covariant co-isometric pair}\},
\end{align*}
$\bullet$ $\fA(\C,\ga,\text{un})_l$: with respect to the collection of representations
\begin{align*}
\{(V\times \pi): (\pi,V) \text{ is a left covariant unitary pair}\}.
\end{align*}
}
\textup{For the right-covariant case, we define the enveloping operator algebras $\fA(\C,\ga,\text{contr})_r$, $\fA(\C,\ga,\text{is})_r$, $\fA(\C,\ga,\text{co-is})_r$ and $\fA(\C,\ga,\text{un})_r$ of $\ell^1(\bbZ_+,\C,\ga)_r$, analogously. }
\end{definition}

As we will see, the above cases sum up to 2 different classes, the left isometric and the right isometric case. This is done by using a two-step method. First, we use a duality between the left and the right cases, which we establish below. Secondly we use dilation theory \cite{MuhSol06, Sta93} to identify completely isometrically the operator algebras, as we explain in the appropriate sections.

\begin{remark}
\textup{There is a bijection between the left covariant pairs and the right covariant pairs. More precisely $(\pi,V)$ is a left covariant contractive (resp. isometric, co-isometric, unitary) pair if, and only if, $(\pi,V^*)$ is a right covariant contractive (resp. co-isometric, isometric, unitary) pair. Indeed, taking adjoints in the relation (\ref{eq:left covariance}) implies $V^*\pi(c^*)=\pi(\ga(c^*))V^*$, hence $V^*\pi(c)=\pi(\ga(c))V^*$, for any $c\in \C$, since $\C$ is selfadjoint.}

\textup{However, the map $c\mapsto \pi(c^*)$ is not a $*$-homomorphism and we cannot pass from the representations of $\ell^1(\bbZ_+,\C,\ga)_l$ to the representations of $\ell^1(\bbZ_+,\C,\ga)_r$ simply by taking adjoints. Nevertheless, the following trick, establishes a duality that simplifies our proofs.}
\end{remark}

For convenience, we use the symbol $\F_{t,l}$, $t=1,2,3,4$, for the collection of the left covariant contractive, isometric, co-isometric and unitary pairs, respectively. Also, we use the symbol $\F_{t,r}$, $t=1,2,3,4$, for the right covariant contractive,  co-isometric, isometric and unitary pairs, respectively. We define the antilinear bijection
\begin{align*}
^\#\colon c_{00}(\bbZ_+)\odot \C \rightarrow \C \odot c_{00}(\bbZ_+),
\end{align*}
so that $(\gd_n\otimes c)^\#= c^*\otimes \gd_n$, for every $c\in \C$. Abusing notation we write $(F^\#)^\#=F$, for every $F\in c_{00}(\bbZ_+) \odot \C$. This bijection is an isometry and extends to an isometry of $\ell^1(\bbZ_+,\C,\ga,)_l$ onto $\ell^1(\bbZ_+,\C,\ga,)_r$, for which we use the same symbol. Moreover, for every representation $(H,\rho)$ of $\ell^1(\bbZ_+,\C,\ga)_r$ we define
\begin{align*}
\rho^\#\colon \ell^1(\bbZ_+,\C,\ga)_l\rightarrow \B(H): \rho^\#(F)=\rho(F^\#)^*.
\end{align*}
It is routine to see that $(H,\rho^\#)$ is a $|\cdot|_1$-representation of $\ell^1(\bbZ_+,\C,\ga)_l$, and that $\rho\in \F_{t,r}$ if and only if $\rho^\#\in \F_{t,l}$.

\begin{lemma}\label{duality}
Let $\rho\in \F_{t,r}$, $t=1,2,3,4$ and $[F_{i,j}]\in \M_\nu(\ell^1(\bbZ_+,\C,\ga)_r)$, $\nu\geq 1$. Then $\rho^\# \in \F_{t,l}$ and
\begin{align*}
\nor{[\rho(F_{i,j})]}_{\B(H^\nu)}
  =
\nor{[\rho^\#(F_{i,j}^\#)]}_{\B(H^\nu)}.
\end{align*}
\end{lemma}
\begin{proof}
Recall that the transpose map $A\mapsto A^t$ is isometric; hence for every $[F_{ij}]\in \M_\nu(\ell^1(\bbZ_+,\C,\ga)_r)$, we have
\begin{align*}
\nor{\left[\rho^\#(F_{ij}^\#)\right]}
& =
\nor{\left[\rho\big((F_{ij}^\#)^\#\big)^*\right]}
  =
\nor{\big[\rho(F_{ij})^*\big]}\\
& =
\nor{\big[\rho(F_{ij})^*\big]^*}
  =
\nor{\big[\rho(F_{ji})\big]}
  =
\nor{\big[\rho(F_{ij})\big]^t}
  =
 \nor{\big[\rho(F_{ij})\big]},
\end{align*}
and the proof is complete.
\end{proof}

\begin{remark}\label{r:pet}
\textup{One has to be careful with the connection between the left case and the right case. Following \cite[Remark 4]{Petarx} we get that even if $\ga$ is a $*$-automorphism and $\C$ is a commutative $\ca$-algebra then $\fA(\C,\ga,\text{is})_l$ is not always isometrically isomorphic to $\fA(\C,\ga,\text{is})_r$, otherwise $\ga$ would be always conjugate to its inverse. For a counterexample see \cite{HoaPar66}.}
\end{remark}

\section{Semicrossed products over left co-isometric and left unitary
covariant pairs}\label{right}

Let us start by examining the semicrossed products $\fA(\C,\ga,\text{co-is})_l$ and $\fA(\C,\ga,\text{un})_l$. We show that they are completely isometrically isomorphic and that their $\ca$-envelope is a crossed product. To do so, it is easier first to consider the enveloping operator algebras $\fA(\C,\ga,\text{is})_r$ and $\fA(\C,\ga,\text{un})_r$, and then use Lemma \ref{duality} to pass to the left case.

By Definition \ref{main def}, $\fA(\C,\ga,\text{is})_r$ is the enveloping operator algebra of the Banach algebra $\ell^1(\bbZ_+,\C,\ga)_r$ with respect to the representations $(\pi\times V)$, where $(\pi,V)$ is a right covariant isometric pair. For every $\nu\geq 1$ and $[F_{ij}] \in \M_\nu(\ell^1(\bbZ_+,\C,\ga)_r)$, let the seminorms
\begin{align*}
\om_\nu([F_{i,j}])
 =
\sup \left\{ \nor{[(\pi\times V)(F_{i,j})]}_{\B(H^\nu)}: (\pi,V) \text{ r.cov.isom. pair} \right\}.
\end{align*}
and let $\N=\{F\in \ell^1(\bbZ_+,\C,\ga)_r: \om_1(F)=0\}$. Then $\om_1$ induces a norm on the quotient $\ell^1(\bbZ_+,\C,\ga)_r/\N$, given by $\nor{F+\N}_\infty:= \om_1(F)$.

An analogous procedure is followed for the definition of the semicrossed product $\fA(\C,\ga,\text{un})_r$.

\begin{proposition}\label{right are the same}
The semicrossed products $\fA(\C,\ga,\text{is})_r$ and $\fA(\C,\ga,\text{un})_r$ are completely isometrically isomorphic.
\end{proposition}
\begin{proof}
Since every right covariant unitary pair is a right covariant isometric pair, it suffices to prove that every right covariant isometric pair dilates to a right covariant unitary pair. But this is established in the proof of \cite[Proposition 2.3]{Sta93}.
\end{proof}

\begin{remark}\label{remark radical}
\textup{The seminorms $\om_\nu$ are not norms in general. For example, assume that the $*$-homomorphism $\ga$ has non-trivial kernel and let $c\in \ker\ga$. Then $V\pi(c)=\pi(\ga(c))V=0$ for every right covariant isometric pair $(\pi,V)$. Since $V$ is an isometry we get that $\pi(c)=0$. Hence $\om_1(c \otimes \gd_0)=0$. Note that the same holds for every $c\in \R_\ga$.}
\end{remark}

The next proposition shows the connection between the radical $\R_\ga$ and the kernel $\N$. In its proof we prove also the existence of a non-trivial right covariant unitary pair. As a consequence $\fA(\C,\ga,\text{is})_r$ and $\fA(\C,\ga,\text{un})_r$ are not zero.

\begin{proposition}
Let $\N=\{F\in \ell^1(\bbZ_+,\C,\ga)_r: \om_1(F)=0\}$. Then $\N=\ell^1(\bbZ_+,\R_\ga,\ga)_r$.
\end{proposition}
\begin{proof}
First of all, note that $\ell^1(\bbZ_+,\R_\ga,\ga)_r$ is a $|\cdot|_1$-closed ideal of the algebra $\ell^1(\bbZ_+,\C,\ga)_r$ that is contained in $\N$. Now, consider the crossed product $\C_\infty \rtimes_{\ga_\infty} \bbZ$ (see the definitions in the introduction). Then the right regular representation $(\pi\times U^*)$ of the crossed product induces a right covariant unitary pair for $(\C,\ga)$. Indeed, it suffices to prove that $\pi$ induces a representation of $\C$. Let $q\colon \C \rightarrow \C/\R_\ga$ be the canonical $*$-epimorphism and recall that $\C/\R_\ga$ embeds in $\C_\infty$. Thus $(\pi\circ q, U^*)$ is a right covariant unitary pair of $(\C,\ga)$.

Let $F \in \N$ with $F=|\cdot|_1-\lim_N \sum_{n=0}^N c_n\otimes \gd_n$. Then
\begin{align*}
\lim_N \sum_{n=0}^N \pi (q(c_n))U^{-n}
& =
\lim_N \sum_{n=0}^N \left((\pi\circ q)\times U^* \right) (c_n\otimes \gd_n)\\
& =
\left((\pi\circ q)\times U^*\right)(F).
\end{align*}
But $F\in \N$, hence $\left((\pi\circ q)\times U^*\right)(F)=0$. For $\xi,\eta \in H$ we have
\begin{align*}
\sca{\pi\circ q(c_n)(\xi),\eta}
& =
\lim_N \sum_{k=0}^N \sca{\pi (q(c_k))U^{-k}(\xi \otimes e_n),\eta\otimes e_0}\\
& =
\sca{\big((\pi\circ q)\times U^*\big)(F)(\xi \otimes e_n),\eta\otimes e_0}=0,
\end{align*}
hence $\pi(q(c_n))=0$, so $q(c_n)=0$. Thus $c_n\in \R_\ga$, for every $n\geq 0$.
\end{proof}

For the next Proposition, recall that we can define the injective $*$- homomorphism $\dot{\ga}\colon \C/\R_\ga \rightarrow \C/\R_\ga$, with $\dot{\ga}(c+\R_\ga)= \ga(c)+ \R_\ga$.

\begin{proposition}\label{right isomorphic}
The semicrossed product $\fA(\C,\ga,\text{is})_r$ is completely isometrically isomorphic to the semicrossed product $\fA(\C/\R_\ga,\dot{\ga},\text{is})_r$.
\end{proposition}
\begin{proof}
It suffices to show that the map
\begin{align*}
Q\colon \ell^1(\bbZ_+,\C,\ga)_r /\N
& \rightarrow
\ell^1(\C/\R_\ga, \dot{\ga}, \bbZ_+)_r \\
c\otimes \gd_n + \N \hspace{2mm}
& \mapsto
\hspace{2mm} (c+\R_\ga)\otimes \gd_n,
\end{align*}
is completely isometric. To this end, let $F=\sum_{n=0}^k c_n\otimes \gd_n$ and $G=Q(F+\N)=\sum_{n=0}^k (c_n+\R_\ga)\otimes \gd_n$. If $(\pi,V)$ is a right covariant isometric pair for $\ell^1(\bbZ_+,\C,\ga)_r$ acting on $H$, then $\pi(\R_\ga)H=0$ by Remark \ref{remark radical}. Thus $\pi$ induces a representation $\sigma\colon \C/\R_\ga \rightarrow \B(H)$ with $\sigma(c + \R_\ga)=\pi(c)$. Hence $(\sigma,V)$ is a right covariant isometric pair of $\ell^1(\C/\R_\ga,\dot{\ga},\bbZ_+)_r$ and $(\pi\times V)(F)=(\sigma \times V)(G)$. Thus,
\begin{align*}
\nor{(\pi\times V)(F)} = \nor{(\sigma\times V)(G)} \leq \nor{G}_\infty.
\end{align*}
Hence $\om_1(F)\leq \nor{G}_\infty$ and so $\nor{F+\N}_\infty=\om_1(F) \leq \nor{G}_\infty.$

On the other hand, let $(\rho,V)$ be a right covariant isometric pair of the algebra $\ell^1(\C/\R_\ga,\dot{\ga},\bbZ_+)_r$; then the map
\begin{align*}
\pi: = \rho \circ q\colon \C\rightarrow \B(H) : c \mapsto \rho(c+ \R_\ga)
\end{align*}
is a representation of $\C$. It is easy to see that $(\pi,V)$ is a right covariant isometric pair of $\ell^1(\bbZ_+,\C,\ga)_r$ and that $(\pi\times V)(F)=(\rho\times V)(G)$. Hence,
\begin{align*}
\nor{(\rho\times V)(G)} = \nor{(\pi\times V)(F)} \leq \om_1(F) = \nor{F +\N}_\infty.
\end{align*}
Thus $\nor{G}_\infty \leq \nor{F+\N}_\infty$. The same arguments work for the matricial elements $[F_{ij}]\in \M_\nu(\ell^1(\bbZ_+,\C,\ga)_r)$, $\nu\geq 1$, and the proof is complete.
\end{proof}

Hence, we can assume that $\ga\colon \C \rightarrow \C$ is injective. Then $\om_1=\nor{\cdot}_\infty$ and $\C$ embeds in $\C_\infty$.

\begin{proposition}\label{C*-envelope right isometric}
Let $\ga\colon \C\rightarrow \C$ be an injective $*$-homomorphism. Then, the $\ca$-envelope of $\fA(\C,\ga,\text{un})_r$ is the crossed product $\C_\infty \rtimes_{\ga_\infty} \bbZ$.
\end{proposition}
\begin{proof}
First we show that $\C_\infty \rtimes_{\ga_\infty} \bbZ$ is a $\ca$-cover of $\fA(\C,\ga,\text{un})_r$. It is clear that $\ell^1(\bbZ_+,\C,\ga)_r \subseteq \ell^1(\bbZ,\C_\infty,\ga_\infty)_r$ and that by restricting any right covariant unitary pair of $(\C_\infty,\ga_\infty)$ we get a right covariant unitary pair of $\ell^1(\bbZ_+,\C,\ga)_r$. Thus $  \nor{F}_{\C_\infty \rtimes_{\ga_\infty} \bbZ} \leq \nor{F}_\infty, $ for every $F$ in $\ell^1(\bbZ_+,\C,\ga)_r$.

On the other hand, let $(\rho,U)$ be a right covariant unitary pair of the algebra $\ell^1(\bbZ_+,\C,\ga)_r$. We can extend $(H,\rho)$ to a representation of $\C_\infty$ in the following way: let $x\in \C_\infty$ be such that $\ga_\infty^n(x)\in \C$ for some $n$ and let $\rho'(x)=U^n\rho(\ga_\infty^n(x))(U^*)^n$. It is easy to see that $\rho'(x)$ is independent of the choice of $n$, hence $\rho'$ extends to a representation of $\C_\infty$. Moreover, $(\rho',U)$ is a right covariant unitary pair of $\ell^1(\bbZ,\C_\infty,\ga_\infty)_r$. Indeed, let $x\in \C_\infty$ such that $\ga_\infty^n(x)\in \C$; then $\ga(\ga^n_\infty(x))=\ga_\infty^{n+1}(x)$, since $\ga_\infty$ extends $\ga$. Thus,
\begin{align*}
U\rho'(x)
& =
U \cdot U^n \rho\big(\ga_\infty^n(x)\big) (U^*)^n
  =
U^{n}\cdot U \rho\big(\ga_\infty^n(x)\big)\cdot (U^*)^n \\
& =
U^{n}\cdot \rho\big(\ga(\ga_\infty^n(x))\big) U\cdot (U^*)^n
  =
U^{n} \rho\big(\ga_\infty^{n+1}(x)\big) (U^*)^{n-1}\\
& =
U^{n}\rho\big(\ga_\infty^n(\ga_\infty(x))\big)(U^*)^{n}\cdot U
  =
\rho'\big(\ga_\infty(x)\big)U.
\end{align*}
Hence, for every $F$ in $\ell^1(\bbZ_+,\C,\ga)_r$ and every right covariant unitary pair $(\rho,U)$ of $\ell^1(\bbZ_+,\C,\ga)_r$, we have that
\begin{align*}
\nor{(\rho\times U)(F)}=\nor{(\rho'\times U)(F)} \leq  \nor{F}_{\C_\infty \rtimes_{\ga_\infty} \bbZ},
\end{align*}
so $\nor{F}_\infty \leq \nor{F}_{\C_\infty \rtimes_{\ga_\infty} \bbZ_+}$. Note that the same arguments work for $[F_{ij}]\in \M_\nu(\ell^1(\bbZ_+,\C,\ga)_r)$, $\nu\geq 1$. Hence, if $(\widehat{\pi}\times U)$ is the right regular representation of the crossed product, the map
\begin{align*}
\fA(\C,\ga,\text{un})_r \rightarrow \C_\infty \rtimes_{\ga_\infty} \bbZ:
\sum_{n=0}^k c \otimes \gd_n \mapsto \sum_{n=0}^k \widehat{\pi}(c)U^n
\end{align*}
is a completely isometric homomorphism.

In order to conclude that the crossed product is a $\ca$-cover, it suffices to prove that every element of the form $\widehat{\pi}[0,\dots,c,\ga(c),\dots]$ is in the $\ca$-algebra $\ca(\widehat{\pi},U)$ generated by the range of $(\widehat{\pi}\times U)$. Recall that $\widehat{\pi}(\ga_\infty(x))=U\widehat{\pi}(x)U^*$ for every $x\in \C_\infty$, so $\widehat{\pi}(x)= U\widehat{\pi}(\ga_\infty^{-1}(x)) U^*$, for every $x\in \C_\infty$. Thus
\begin{align*}
\widehat{\pi}[0,c,\ga(c),\dots]
& =
U \widehat{\pi}(\ga_\infty^{-1}[0,c,\dots]) U^* \\
& =
U\widehat{\pi}[c,\ga(c),\dots] U^*\in \ca(\widehat{\pi},U).
\end{align*}
Induction shows that every element of the form $\widehat{\pi}[0,\dots,c,\ga(c),\dots]$ is in $\ca(\widehat{\pi},U)$.

To end the proof, let $\ca_e$ be the $\ca$-envelope of $\fA(\C,\ga,\text{un})_r$ and $\Phi$ be the $*$-epimorphism $ \Phi\colon \C_\infty \rtimes_{\ga_\infty} \bbZ \rightarrow \ca_e$, such that $\Phi(\widehat{\pi}[c,\ga(c),\dots])= c\otimes \gd_0$ for every $c\in \C$. Assume that the \v{S}ilov ideal $\J=\ker\Phi$ is non-trivial. Then it is invariant by the gauge action for the right regular representation, hence it has non-trivial intersection with the fixed point algebra of the gauge action, which is exactly $\C_\infty$. So, there is an $n$ such that $\J\cap \ga_\infty^{-n}(\C) \neq (0)$. Let $0\neq c\in \C$ such that $\widehat{\pi}[0,\dots,c,\dots] \in \J$. Then $\widehat{\pi}[c,\ga(c),\dots]=(U^*)^{n}\widehat{\pi}[0,\dots,c,\dots]U^{n} \in \J$, so $\J\cap \C\neq (0)$. But then
\begin{align*}
0 = \nor{\widehat{\pi}[c,\ga(c),\dots] +\J}
  =
\nor{\Phi(\widehat{\pi}[c,\ga(c),\dots])}
  =
\nor{c \otimes \gd_0}_\infty=\nor{c}_\C,
\end{align*}
which is a contradiction, since $\C$ is contained isometrically in the semicrossed product. Thus $\J=(0)$.
\end{proof}

Using Propositions \ref{right are the same}, \ref{right isomorphic} and \ref{C*-envelope right isometric}, we arrive to the following Theorem for the general case.

\begin{theorem}\label{theorem C^*-envelope right isometric}
Let $\ga\colon \C\rightarrow \C$ be a $*$-homomorphism. Then the $\ca$-envelope of the semicrossed products $\fA(\C,\ga,\text{is})_r$ and $\fA(\C,\ga,\text{un})_r$ is the crossed product $(\C/\R_\ga)_\infty \rtimes_{(\dot{\ga})_\infty} \bbZ$.
\end{theorem}
\begin{proof}
By Propositions \ref{right isomorphic} and \ref{right are the same} we have that $$\fA(\C,\ga,\text{is})_r \simeq \fA(\C/ \R_\ga,\dot{\ga},\text{is})_r \simeq \fA(\C/ \R_\ga,\dot{\ga},\text{un})_r,$$ thus they have the same $\ca$-envelope. By Proposition \ref{C*-envelope right isometric} this is the crossed product $(\C/\R_\ga)_\infty \rtimes_{(\dot{\ga})_\infty} \bbZ$.
\end{proof}

Now, for the left case, let $\N_{3,l}$ be the kernel with respect to the collection $\F_{3,l}$. Then $\N_{3,l}=\N$, by Lemma \ref{duality}
and because $\R_\ga$ is self-adjoint. By the same Lemma,
\begin{align*}
\nor{[F_{ij}+\N]}_\nu=\nor{[F_{ij}^\# +\N]}_\nu,
\end{align*}
for every $[F_{ij}] \in \M_\nu(\ell^1(\bbZ_+,\C,\ga)_l)$, $\nu\geq 1$. Hence, we get the next proposition.

\begin{proposition}\label{left isomorphic 1}
The semicrossed product $\fA(\C,\ga,\text{co-is})_l$ is completely isometrically isomorphic to the semicrossed product $\fA(\C/\R_\ga, \dot{\ga},\text{co-is})_l$.
\end{proposition}

\begin{theorem}\label{C*-envelope left co-isometric}
Let $\ga\colon \C\rightarrow \C$ be a $*$-homomorphism. Then the $\ca$-envelope of the semicrossed products $\fA(\C,\ga,\text{co-is})_l$ and $\fA(\C,\ga,\text{un})_l$ is the crossed product $(\C/\R_\ga)_\infty \rtimes_{(\dot{\ga})_\infty} \bbZ$.
\end{theorem}
\begin{proof}
Without loss of generality, we can assume that $\ga\colon \C\rightarrow \C$ is injective. We will show that there is a completely isometric homomorphism of $\fA(\C,\ga,\text{co-is})_l$ into $\C_\infty\rtimes_{\ga_\infty}\bbZ$. Let $[F_{ij}] \in \M_\nu\left((\ell^1(\bbZ_+,\C,\ga)_l\right)$ and let $(\pi,V)$ be a left covariant co-isometric pair. Then, $(\pi,V^*)$ is a right covariant isometric pair, hence it dilates to a right covariant unitary pair $(\Pi,U)$ of $(\C_\infty,\ga_\infty)$. Thus, $(\Pi\times U)$ is a dilation of the representation $(\pi\times V^*)=(V\times \pi)^\#$. So, the pair $(\Pi,U^*)$ is a left covariant unitary pair of $(\C_\infty,\ga_\infty)$, hence $(\Pi\times U)^\#$ is a representation of the crossed product $\C_\infty \rtimes_{\ga_\infty} \bbZ$. Thus, we have that
\begin{align*}
\|[ (V\times\pi) (F_{ij})]\|
& =
\nor{[(V\times \pi)^\#(F_{ij}^\#)]}
 \leq
\nor{[(\Pi\times U^*)(F_{ij}^\#)]}\\
& =
\nor{[(\Pi \times U)^\#(F_{ij})]}
 \leq
\nor{[F_{ij}]}_{\M_\nu(\C_\infty \rtimes_{\ga_\infty}
 \bbZ)}.
\end{align*}
Thus $\nor{[F_{ij}]}_\nu \leq \nor{[F_{ij}]}_{\M_\nu(\C_\infty \rtimes_{\ga_\infty} \bbZ)}$. Moreover one can easily check that $\nor{[F_{ij}]}_{\M_\nu(\C_\infty \rtimes_{\ga_\infty} \bbZ)} \leq \nor{[F_{ij}]}_\nu $. Hence, the identity map $\ell^1(\bbZ_+,\C,\ga)_l \hookrightarrow \C_\infty \rtimes_{\ga_\infty} \bbZ$ is a completely isometric homomorphism. Since we have assumed that $\ga$ is injective, then $\N_{3,l}=(0)$. Thus the identity map extends to a completely isometric homomorphism of $\fA(\C,\ga,\text{co-is})_l$. The rest of the proof goes in a similar way with the one of Proposition \ref{C*-envelope right isometric}.
\end{proof}

\begin{remark}\label{rem exe}
\textup{From a first look, it seems that the theory we used here could be related to the crossed product by an endomorphism $\T(\C,\ga,\L)$ examined in \cite{Exel03}, but there is a significant difference. In \cite[Definition 3.1]{Exel03} Exel considers a universal $\ca$-algebra for which the representation theory consists of pairs $(\pi,S)$ such that $S\pi(c)=\pi(\ga(c))S$, but also $S^*\pi(c) S=\pi(\L(c))$, for a chosen transfer operator $\L$. On the other hand the pairs we examined do define a transfer operator on every Hilbert space (just by defining $\L(\pi(c))=S^* \pi(c) S$), but $\L$ is not the same for every such pair.}

\textup{Nevertheless, our theory is applicable to $\U(\C,\ga)$ defined in \cite[Definition 4.4]{Exel03}. Let the operator algebra generated by the analytic polynomials in $\U(\C,\ga)$. Then, every representation consists of pairs $(\pi,S)$ where $S$ is an isometry, subject to the relation $\pi(\ga(c))=S \pi(c) S^*$. Hence, for any such a pair we get that $\pi(\ga(c))S=S\pi(c)$ and we can apply the methods of Theorem \ref{theorem C^*-envelope right isometric}. In particular, one can show that the $\ca$-envelope of the non-selfadjoint part of $\U(\C,\ga)$ generated by polynomials is the crossed product $\C_\infty \rtimes_{\ga_\infty} \bbZ$.}

\textup{The same holds for $\C \rtimes_{\ga,\L} \bbN$ \cite[Definition 3.7]{Exel03} when $\ga(\C)$ is hereditary and $\ga$ is $*$-injective, because of \cite[Theorem 4.7]{Exel03}.}
\end{remark}

\section{Semicrossed products over left contractive and left isometric covariant pairs}\label{left}

We recall that $\fA(\C,\ga,\text{contr})_l$ is the enveloping operator algebra of the algebra $\ell^1(\bbZ_+,\C,\ga)_l$ with respect to the family of the representations $(V\times \pi)$, where $(\pi,V)$ ranges over left covariant contractive pairs. For every $\nu\geq 1$ and $[F_{ij}] \in \M_\nu(\ell^1(\bbZ_+,\C,\ga)_r)$, we define the seminorms
\begin{align*}
\nor{[F_{i,j}]}_\nu = \sup\{ \nor{[(\pi\times V)(F_{i,j})]}_{\B(H^\nu)}: (\pi,V) \text{ r.cov.contr. pair}\}.
\end{align*}

\begin{proposition}\label{left isomorphic 2}
The semicrossed products $\fA(\C,\ga,\text{contr})_l$ and $\fA(\C,\ga,\text{is})_l$ are completely isometrically isomorphic.
\end{proposition}
\begin{proof}
Since every left covariant isometric pair is a left covariant contractive pair, it suffices to prove that every left covariant contractive pair dilates to a left covariant isometric pair. But this is established in \cite{MuhSol06}.
\end{proof}

The following is an example of a left covariant isometric pair. It follows that the semicrossed products $\fA(\C,\ga,\text{contr})_l$ and $\fA(\C,\ga,\text{is})_l$ are not zero; moreover, that the $\nu$-seminorms are in fact norms. In Theorem \ref{left only one} we show that the following construction gives a completely isometric representation of $\fA(\C,\ga,\text{contr})_l$ and $\fA(\C,\ga,\text{is})_l$.

\begin{example}\label{example left pure isometric}
\textup{Let $(H_0,\pi)$ be a representation of $\C$. We define $\widetilde{\pi}(c) =\diag\{\pi(\ga^n(c)):n\in\bbZ_+\}$, and $S=I_{H_0}\otimes s$, where $s$ is the unilateral shift, acting on the Hilbert space $H_0\otimes\ell^2(\bbZ_+)$. Then $(\widetilde{\pi},S)$ is a left covariant isometric pair.}

\textup{Moreover, \emph{if $(H_0,\pi)$ is faithful, then the induced representation $(S\times \widetilde{\pi})$ is faithful on $\ell^1(\bbZ_+,\C,\ga)_l$}. Indeed, let $F=|\cdot|_1-\lim_N \sum_{n=0}^N \gd_n\otimes c_n$, such that $(S\times \widetilde{\pi})(F)=0$. Then, for every $\xi,\eta \in H_0$,
\begin{align*}
\sca{\pi(c_n)\xi,\eta}
& =
\lim_N \sum_{k=0}^N \sca{S^k\widetilde{\pi}(c_k)(\xi\otimes e_0), \eta\otimes e_n} \\
& =
\sca{(S\times \widetilde{\pi})(F)(\xi\otimes e_0), \eta\otimes e_n}=0.
\end{align*}
Hence $c_n=0$  for every $n$, since $\pi$ is faithful, thus $F=0$.}
\end{example}

For every left covariant isometric pair $(\pi,V)$ we denote by $\ca(\pi,V)$ the $\ca$-algebra generated by the range of the representation $(V\times \pi)$. Due to the left covariance relation, $\ca(\pi,V)$ is the closure of the polynomials
\begin{align*}
\sum_{n,m} V^n \pi(c_{n,m}) (V^*)^m,\ \ c_{n,m}\in \C,\ \ n,m \in \bbZ_+.
\end{align*}

Let $H_u=\oplus_i H_i$, $\pi_u= \oplus_i \pi_i$ and $V_u=\oplus_i V_i$, where the summand ranges over the left covariant isometric pairs $(\pi_i,V_i)$ that act on Hilbert spaces $H_i$. Then the semicrossed product $\fA(\C,\ga,\text{is})_l$ is the closure of the polynomials
$\sum_{n=0}^k V_u^n \pi_u(c_n),\ c_n\in \C,\ k\in \bbZ_+$.

The $\ca$-algebra $\ca(\pi_u,V_u)$ has the following universal property: for every left covariant isometric pair $(\pi,V)$ there is a $*$-epimorphism $\Phi\colon  \ca(\pi_u,V_u) \rightarrow \ca(\pi,V)$, such that $\Phi\circ \pi_u = \pi$ and $\Phi\circ V_u= V$. (The $*$-epimorphism $\Phi$ is induced by restricting $\pi_u$ and $V_u$ to $H\subseteq H_u$.)

For any $z\in \bbT$ we define a $*$-automorphism $\beta_z$ of $\ca(\pi_u,V_u)$, such that $\beta_z(\pi_u(c))=\pi_u(c)$, $c\in \C$ and $\beta_z(V^n_u)=z^nV^n_u$, $n\in \bbZ_+$. An $\eps/3$-argument, along with the fact that $\ca(\pi_u,V_u)$ is the closed linear span of the monomials $V_u^n\pi(c) (V_u^*)^m$, shows that the family $\{\beta_z\}_{z\in \bbT}$ is point-norm continuous. Thus, we can define the conditional expectation $\E\colon \ca(\pi_u,V_u)\rightarrow \ca(\pi_u,V_u)$, by $\E(F):=\int_\bbT \beta_z(F) dz,\ \ F\in \ca(\pi_u,V_u)$, where $dz$ is Haar measure on the unit circle $\bbT$. The fixed point algebra $\ca(\pi_u,V_u)^{\beta}$, i.e. the range of $\E$, is the closed linear span of $\sum_{n=0}^k V_u ^n\pi_u(c_n) (V_u^*)^n$, $c_n\in \C$. Hence, the fixed point algebra is (the inductive limit) $\overline{\cup_k B_k}$, where $B_k= \overline{\Span}\{ \sum_{n=0}^k V_u^n\pi_u(c_n) (V_u^*)^n: c_n \in \C\}$. It is a routine to check that $\E$ is a norm-continuous, faithful projection onto the fixed point algebra.\\

Now, fix a faithful representation $(H_0,\pi)$ of $\C$, and let $(\widetilde{\pi},S)$ be as in Example \ref{example left pure isometric}. For every $z\in \bbT$ we define the unitary operator $u_z\colon \ell^2(\bbZ_+)\rightarrow\ell^2(\bbZ_+)$, by $u_z(e_n)=z^n e_n$. Let $U_z=1_{H_0} \otimes u_z$; the map $\gamma_z=ad_{U_z}$ satisfies $\gamma_z(\widetilde{\pi}(c))=\widetilde{\pi}(c)$, $c\in \C$ and $\gamma_z(S^n)=z^nS^n$, $n\in \bbZ_+$, hence defines a $*$-automorphism of $\ca(\widetilde{\pi}, S)$. Again, an $\eps/3$-argument shows that $\{\gamma_z\}_{z\in \bbT}$ is a point-norm continuous family, hence induces the conditional expectation $E(F):=\int_\bbT \gamma_z(F) dz$, $F\in \ca(\widetilde{\pi}, S)$. The map $E$ is a norm-continuous faithful projection onto the fixed point algebra $\ca(\widetilde{\pi}, S)^\gamma = \overline{\Span}\{ ^k S^n\pi(c) (S^*)^n: c\in \C\}$.

For the canonical $*$-epimorphism $\Phi\colon \ca(\pi_u,V_u) \rightarrow \ca(\widetilde{\pi}, S)$ we obtain that $\Phi\circ \beta_z= \gamma_z \circ \Phi$, hence $\Phi\circ \E= E \circ \Phi$. Thus the restriction of $\Phi$ on $\ca(\pi_u,V_u)^\beta$ is a $*$-homomorphism onto $\ca(\widetilde{\pi}, S)^\gamma$. The following is a revision of \cite[Theorem 1.4]{Kak09}.

\begin{theorem}\label{left only one}
The $*$-epimorphism $\Phi\colon \ca(\pi_u,V_u)\rightarrow \ca(\widetilde{\pi},S)$ is injective, hence a $*$-isomorphism.
As a consequence, the semicrossed products $\fA(\C,\ga,\text{contr})_l$ and $\fA(\C,\ga,\text{is})_l$ are completely isometrically isomorphic to the closure of the polynomials $\sum_{n=0}^k S^n\widetilde{\pi}(c_n)$, $c_n\in \C$, $($in $\ca(\widetilde{\pi},S) )$, for any faithful representation $(H_0,\pi)$ of $\C$.
\end{theorem}
\begin{proof}
First we show that the restriction of $\Phi$ to $\ca(\pi_u,V_u)^\beta$ is injective. Assume that $\ker(\Phi|_{\ca(\pi_u,V_u)^\beta})$ is not trivial. Since $\ca(\pi_u,V_u)^\beta$ is an inductive limit there is a $k$ such that $B_{k} \cap \ker(\Phi|_{\ca(\pi_u,V_u)^\beta}) \neq (0)$, thus there is a non-zero analytic polynomial $F=\sum_{n=0}^k V_u^n\widetilde{\pi_u}(c_n)(V^*_u)^n$ such that $\Phi(F)= \sum_{n=0}^k S^n\widetilde{\pi}(c_n)(S^*)^n =0$. Note that
\[
S^n\widetilde{\pi}(c)(S^*)^n
=
\diag \{\underbrace{0,\dots,0}_{\text{n-times}},\pi(c), \pi(\ga(c)),\dots\},
\]
hence the $(m,m)$-element $(\Phi(F))_{m,m}$ of $\Phi(F)$ equals to
\begin{align*}
(\Phi(F))_{m,m}&=
\begin{cases}
   \pi \left(\ga^m(c_0)+\ga^{m-1}(c_1)+\dots +c_m \right)
   &,\text{for}\ m<k,\\
   \pi \left(\ga^m(c_0)+\ga^{m-1}(c_1)+\dots
+\ga^{m-k}(c_k) \right) &,\text{for}\ m\geq k,
\end{cases}
\\
& =
\pi \left( \sum_{j=0}^{\min\{m,k\}} \ga^{m-j}(c_{m-j}) \right).
\end{align*}
Since $\Phi(F)=0$, we obtain $(\Phi(F))_{0,0}=0$, hence $\pi(c_0)=0$. Since $(H_0,\pi)$ is injective we get that $c_0=0$. Thus $(\Phi(F))_{1,1}= \pi(c_1)$ and arguing as previously we obtain $c_1=0$. The repetition of the argument shows that $c_m=0$ for all $m=0,1,\dots,k$, hence $F= \sum_{n=0}^k V_u^n\widetilde{\pi_u}(c_n)(V^*_u)^n=0$, which is a contradiction. Thus the restriction of $\Phi$ to $\ca(\pi_u,V_u)^\beta$ is injective.

Now, let $F\in\ker\Phi$, then $F^*F\in\ker\Phi$. Hence, $\Phi\circ\E(F^*F)=E\circ \Phi(F^*F)=0$. But, $\E(F^*F)\in \ca(\pi_u,V_u)^\beta$, and the restriction of $\Phi$ to $\ca(\pi_u,V_u)^\beta$ is injective; thus $\E(F^*F)=0$. So, $F=0$, since $\E$ is faithful.
\end{proof}

The next Theorem follows by Lemma \ref{duality} as used in Proposition \ref{C*-envelope left co-isometric}.

\begin{theorem}
The semicrossed products $\fA(\C,\ga,\text{contr})_r$ and $\fA(\C,\ga,\text{co-is})_r$ are completely isometrically isomorphic to the closed linear span of the polynomials $\sum_{n=0}^k \widetilde{\pi}(c_n)(S^*)^n$, $c_n\in \C$, for any faithful representation $(H,\pi)$ of $\C$ and $(\widetilde{\pi},S)$ as in the Example \ref{example left pure isometric}.
\end{theorem}

We now proceed to the determination of the $\ca$-envelope of the semicrossed products $\fA(\C,\ga,\text{contr})_l$ and $\fA(\C,\ga,\text{is})_l$. As mentioned in the introduction, in \cite{Petarx} Peters computes the $\ca$-envelope of $\fA(\C,\ga,\text{is})_l$, when $\C$ is commutative and $\ga$ is injective. In \cite{KakKat10} we prove a similar Theorem without the assumption of commutativity, but still assuming that $\ga$ is a $*$-automorphism. In \cite{KakKat11} with Elias Katsoulis we gave the result for the general case in the context of $\ca$-correspondences, by extending the method of ``adding tails'' introduced in \cite{MuTom04}. In particular \cite[Proposition 3.12]{KakKat11} shows the necessity of that extension.\\

Let $\M\equiv \M(\ker\ga)$ be the multiplier algebra of $\ker\ga$, and $\theta\colon \C \rightarrow \M$ be the unique unital $*$-homomorphism extending the natural embedding $\ker\ga\hookrightarrow \M$. Also, consider the $\ca$-algebra $T=c_0(\theta(\C))$; we use the letters $\tx, \ty$, etc. for the elements $(x_n), (y_n) \in T$ and the symbol $\0$ for the zero sequence $(0)\in T$. For the $\ca$-algebra $\B=\C\oplus T$ we define the map $\beta\colon \B\rightarrow \B$ by
\begin{align*}
\beta(c,\tx)\equiv \beta(c,(x_n))=(\ga(c), \theta(c), x_1,x_2,\dots)\equiv (\ga(c),\theta(c),\tx),
\end{align*}
for every $c\in \C$ and $\tx\equiv (x_n)\in c_0(\theta(\C))$. Note that $\B$ contains $\C$, but $\beta$ does not extend $\ga$. Also, $\beta$ is an injective $*$-homomorphism. Indeed, let $(c,\tx)\in\ker\beta$; then $x_n=0$, $\theta(c)=0$ and $\ga(c)=0$. Thus, $c\in\ker\ga$, so $c=\theta(c)=0$. Hence $(c,\tx)=0$. Finally, if $e$ is the unit of $\C$, we have
\begin{equation}\label{eq:unit}
\beta^m(e,\0)(c,\0)
=
(e,\underbrace{1_\M,\dots,1_\M}_{\text{m-times}},0,\dots)(c,\0)
=
(c,\0),
\end{equation}
for every $m\in \bbZ_+$.

For the $*$-automorphism $\beta_\infty\colon \B_\infty \rightarrow \B_\infty$ (which is an extension of $\beta$), consider the crossed product
\begin{align*}
\B_\infty \rtimes_{\beta_\infty} \bbZ
=
\overline{\Span}\{U^n\widehat{\pi}(y): y\in \B_\infty, n\in \bbZ\},
\end{align*}
where $(\widehat{\pi},U)$ is the pair that induces the left regular representation of $(\B_\infty,\beta_\infty)$. Since $\B_\infty = \overline{ \cup_n \beta_\infty^{-n}(\B)}$, and due to the left covariance relation, we get
\begin{align*}
\B_\infty \rtimes_{\beta_\infty} \bbZ
=
\overline{\Span}\{U^n\widehat{\pi}(b)(U^*)^m: n,m\in \bbZ_+, b\in \B\}.
\end{align*}

From now on, fix $\fD$ be the $\ca$-subalgebra of the crossed product generated by $U\widehat{\pi}(e,\0)$ and $\widehat{\pi}(c,\0)$, $c\in \C$. Then $\fD$ is the closed linear span of the monomials $U^n\widehat{\pi}(c,\0)(U^*)^m$, $n,m\in \bbZ_+$, $c\in C$. Also, note that
\begin{equation}\label{eq: u}
U\widehat{\pi}(e,\0)=\widehat{\pi}(e,\0) U\widehat{\pi}(e,\0).
\end{equation}

\begin{lemma}\label{lemma head}\textup{\cite[Lemma 3.4]{KakKat11}}
Every element $U^n\widehat{\pi}(b)(U^*)^n$, $b\in \B$, can be written as $A+\widehat{\pi}(0,\ty)$, where $A\in\fD$.
\end{lemma}
\begin{proof}
For $n=0$, we have $\widehat{\pi}(b)=\widehat{\pi}(c,\tx) =\widehat{\pi}(c,\0)+\widehat{\pi}(0,\tx)$. For $n=1$,
\begin{align*}
U\widehat{\pi}(b)U^*
& =
U\widehat{\pi}(c,\tx)U^*
  =
U\widehat{\pi}(c,\0)U^* + U\widehat{\pi}(0,\tx)U^*.
\end{align*}
Let $c'\in \C$ such that $\theta(c')=x_1$; then
\begin{align*}
\beta_\infty(c',x_2,x_3,\dots)
& =
\beta(c',x_2,x_3,\dots) \\
& =
(\ga(c'),\theta(c'),x_2,\dots)= (\ga(c'),\0)+(0,\tx).
\end{align*}
Thus $(c',x_2,x_3,\dots)=\beta_\infty^{-1}(\ga(c'),\0)+ \beta_\infty^{-1}(0,\tx)$, hence
\begin{align*}
U\widehat{\pi}(0,\tx)U^*
& =
\widehat{\pi}(\beta_\infty^{-1}(0,\tx))
  =
\widehat{\pi}(c',x_2,\dots)-\widehat{\pi} (\beta_\infty^{-1}(\ga(c'),\0))\\
& =
\widehat{\pi}(c',\0) - U\widehat{\pi}(\ga(c'),\0)U^* + \widehat{\pi}(0,x_2,\dots).
\end{align*}
Thus, $U\widehat{\pi}(b)U^*=A+\widehat{\pi}(0,x_2,\dots)$.

The proof is completed by using induction on $n$.
\end{proof}

\begin{proposition}\textup{\cite[Lemma 3.7]{KakKat11}}
The semicrossed product $\fA(\C,\ga,\text{is})_l$ is completely isometrically isomorphic to the closure of the polynomials of the form $\sum_{n=0}^k U^n \widehat{\pi}(c_n,\0)$. Hence, $\fD$ is a $\ca$-cover of $\fA(\C,\ga,\text{is})_l$.
\end{proposition}
\begin{proof}
Let $(H_0,\pi)$ be a faithful representation of $\B_\infty$ and let $(\widehat{\pi}, U)$ be the unitary covariant pair in $H=H_0 \otimes \ell^2(\bbZ)$, that gives the left regular representation of the crossed product. For simplicity, let $\phi$ be the representation of $\C$ given by $\phi(c):= \widehat{\pi}(c,\0)$.
Then $(\phi, U\phi(e))$ is a left covariant contractive pair for $(\C,\ga)$. Indeed,
\begin{align*}
\phi(c) \cdot U \phi(e)
& =
\widehat{\pi}(c,\0)U \cdot \widehat{\pi}((e,\0))
  =
U\widehat{\pi}(\beta_\infty(c,\0))\cdot \widehat{\pi}((e,\0))\\
& =
U\widehat{\pi}\left(\beta(c,\0)(e,\0)\right)
  =
U \widehat{\pi}(\ga(c),\0)
  =
U\phi(e)\cdot \phi(\ga(c)).
\end{align*}
Hence, for every $\nu\geq 1$ and $[F_{ij}] \in \M_\nu(\ell^1(\bbZ_+,\C,\ga))$, we have that
\begin{equation}\label{eq:inequality}
\nor{\big[ \left(U\phi(e)\times \phi\right)(F_{ij})\big]}
\leq
\nor{ [F_{ij}]}_\infty.
\end{equation}
Using equation (\ref{eq: u}), we note that
\begin{align*}
(U\phi(e)\times \phi)(\gd_n\otimes c)
=
(U\phi(e))^n \phi(c)=U^n
\phi(c),
\end{align*}
for every $n\in \bbZ_+$ and $c\in \C$. Thus, we have to prove that, for every $\nu\geq 1$, the inequality (\ref{eq:inequality}) is an equality.

Let $K=[e_n:n\geq 0] \subseteq \ell^2(\bbZ)$. It is easy to see that $H_0\otimes K$ is a reducing subspace for $(H,\phi)$. Then the restriction $\phi|_{H_0\otimes K}$ of $\phi$ to $H_0\otimes K$ is faithful representation, since its $(0,0)$-entry is the faithful representation $\pi$. Moreover, the pair $(P_{H_0\otimes K} U\phi(e) |_{H_0\otimes K}, \phi|_{H_0\otimes K})$ satisfies the left covariance relation for $(\C,\ga)$.
We note that
\begin{align*}
P_{H_0\otimes K}U|_{H_0\otimes K}
=
(1_{H_0}\otimes P_K )(1_{H_0} \otimes u) |_{H_0\otimes K}
=
1_{H_0}\otimes s=S,
\end{align*}
hence, using induction on the relation (\ref{eq: u}), we get that
\begin{align*}
\big(P_{H_0\otimes K} U\phi(e)|_{H_0\otimes K}\big)^n
=
\big(P_{H_0\otimes K} UP_{H_0\otimes K}\big)^n \phi(e)|_{H_0\otimes K}
=
S^n \phi(e)|_{H_0\otimes K}.
\end{align*}

Let $\fC$ be the $\ca$-algebra $\ca\big(\phi|_{H_0\otimes K},P_{H_0\otimes K}U\phi(e)|_{H_0\otimes K}\big)$ and $\Phi$ be the $*$-epimorphism from $\ca(\pi_u,V_u)$ onto $\fC$. We will show that $\Phi$ is injective; as in the proof of Theorem \ref{left only one} it suffices to show that the restriction of $\Phi$ to the fixed point algebra is injective. To this end, for $z\in \bbT$, let $\gm_z=ad_{U_z}$, where $U_z(\xi\otimes e_k)= z^k\xi\otimes e_k$, $k\in \bbZ_+$. Then $\gm_z$ is a $*$-automorphism of $\fC$. Since
\begin{align*}
& \big(P_{H_0\otimes K}\big( U\phi(e)\big)|_{H_0\otimes K}\big)^n \phi(c)|_{H_0\otimes K} \big((P_{H_0\otimes K} \big( U\phi(e)\big)|_{H_0\otimes  K})^*\big)^n\\
& \qquad
= S^n \phi(e) \phi(c)|_{H_0\otimes K} \phi(e) (S^*)^n
= S^n\phi(c)|_{H_0\otimes K} (S^*)^n\\
& \qquad = \diag \{\underbrace{0,\dots,0}_{\text{n-times}}, \pi(c,\0),\pi(\beta_\infty(c,\0)),\dots\},
\end{align*}
we get that the fixed point algebra of $\fC$ is the closed linear span of these monomials. Recall that $(H_0,\pi)$ is a faithful representation of $\C$; hence we can follow mutatis mutandis the arguments of the proof of Theorem \ref{left only one}, and conclude that the restriction of $\Phi$ to the fixed point algebra $\fC^\gm$ is injective

Thus, for every $F\in \ell^1(\bbZ_+,\C,\ga)_l$,
\begin{align*}
\nor{F}_\infty
& =
\nor{ \bigg((P_{H_0\otimes K}U\phi(e)|_{H_0\otimes K})\times \phi|_{H_0\otimes K}\bigg)(F)}.
\end{align*}
But,
\begin{align*}
\bigg((P_{H_0\otimes K}U\phi(e)|_{H_0\otimes K})\times \phi|_{H_0\otimes K}\bigg)
=
P_{H_0\otimes K}\bigg(U\phi(e)\times \phi\bigg)|_{H_0\otimes K}
\end{align*}
and eventually, for any $F\in \ell^1(\bbZ_+,\C,\ga)_l$, we obtain
\begin{align*}
\nor{F}_\infty
& =
\nor{ \bigg((P_{H_0\otimes K}U\phi(e)|_{H_0\otimes K})\times \phi|_{H_0\otimes K}\bigg)(F)}\\
& =
\nor{ P_{H_0\otimes K} \bigg(U\phi(e)\times\phi\bigg)|_{H_0\otimes K}(F)}\\
& \leq
\nor{\big(U\phi(e)\times\phi\big)(F)} \leq \nor{F}_\infty.
\end{align*}
Hence, $\nor{F}_\infty= \nor{\big(U\phi(e)\times \phi\big)(F)}$, for $F\in \ell^1(\bbZ_+,\C,\ga)_l$. The same argument can be used for any matrix $[F_{ij}]$, and the proof is complete.
\end{proof}

Our goal is to show that $\fD$ is a full corner of $\B_\infty \rtimes_{\beta_\infty} \bbZ$. Recall that for a projection $p$ in the multiplier algebra of a $\ca$-algebra $\fC$, the $\ca$-subalgebra $p\fC p$ is called a \emph{corner} of $\fC$. A corner is called \emph{full} if the linear span of $\fC p \fC$ is dense in $\fC$. Equivalently, if $p\fC p$ is not contained in any proper ideal of $\fC$.

Let $p=\widehat{\pi}(e,\0)$. Then it is easy to see that $p$ is a projection in (the multiplier algebra of) $\B_\infty \rtimes_{\beta_\infty} \bbZ$, and that
\begin{enumerate}
\item $p\widehat{\pi}(c,\0)=\widehat{\pi}(c,\0)=\widehat{\pi}(c,\0)p$,
for every $c\in \C$,
\item $p U^n\widehat{\pi}(b)= U^n \widehat{\pi}(c,x_1,\dots,x_n,\0)$,
for every $b=(c,\tx)\in \B$,
\item $U^n\widehat{\pi}(b) p = U^n\widehat{\pi}(c,\0)$,
for every $b=(c,\tx)\in \B$,
\item $p A p= A$, for every $A\in \fD$; hence $p\fD p=\fD$.
\end{enumerate}

\begin{proposition}\textup{\cite[Theorem 3.10]{KakKat11}}
The corner $p(\B_\infty \rtimes_{\beta_\infty} \bbZ)p$ is full and equals to $\fD$.
\end{proposition}
\begin{proof}
Let $U^n\widehat{\pi}(b)(U^*)^m$ be a monomial in the crossed product. If $n\geq m$, then by Lemma \ref{lemma head}
\begin{align*}
p\big(U^n\widehat{\pi}(b)(U^*)^m\big)p
& =
p\big(U^{n-m} U^m \widehat{\pi}(b)(U^*)^m\big)p \\
& =
p\bigg(U^{n-m}\big(A+\widehat{\pi}(0,\ty)\big)\bigg)p=pU^{n-m}Ap=U^{n-m}A,
\end{align*}
for some $A\in \fD$. In the same way we get that $p(U^n\widehat{\pi}(b)(U^*)^m)p=A(U^*)^{m-n}$, for some $A\in \fD$, when $n<m$. Hence, $p(\B_\infty \rtimes_{\beta_\infty} \bbZ)p \subseteq \fD$. On the other hand, $\fD \subseteq \B_\infty \rtimes_{\beta_\infty} \bbZ$, hence $\fD=p\fD p \subseteq p(\B_\infty \rtimes_{\beta_\infty} \bbZ) p$. Thus $\fD$ is the corner $p(\B_\infty \rtimes_{\beta_\infty} \bbZ) p$.

To prove that it is also full, let $\I$ be an ideal in the crossed product, such that $\fD \subseteq \I$. We will prove that $\I$ is not non-trivial. To this end, it suffices to prove that $\widehat{\pi}(\B_\infty)\subseteq \I$. Since $\B_\infty =\overline{\cup\beta_\infty^{-n}(\B)}$ and due to the left covariance relation, it suffices to show that $\widehat{\pi}(\B)\subseteq \I$. First, $\widehat{\pi}(\C)\in \fD \subseteq \I$. In order to prove that $\widehat{\pi}(c_0(\theta(\C))) \subseteq \I$, it suffices to show that $\widehat{\pi}(0,\underbrace{0,\dots,0}_{\text{n-times}},1_\M,0,\dots)\in \I$, for every $n\in\bbZ_+$. Note that
\begin{align*}
\widehat{\pi}(0,\underbrace{0,\dots,0}_{\text{n-times}},1_\M,0,\dots)
& =
\widehat{\pi}(\beta^n(0,1_\M,\0))\\
& =
\widehat{\pi}\Big((\beta_\infty)^n(0,1_\M,\0)\Big)
  =
(U^*)^n \widehat{\pi}(0,1_\M,\0)U^n.
\end{align*}
Hence, it suffices to show that $\widehat{\pi}(0,1_\M,\0)\in \I$.
Indeed,
\begin{align*}
\widehat{\pi}(0,1_\M,\0)
& =
\widehat{\pi}(\beta(e,\0)-(e,\0))
  =
\widehat{\pi}(\beta(e,\0))- \widehat{\pi}(e,\0)\\
& =
\widehat{\pi}(\beta_\infty(e,\0))-\widehat{\pi}(e,\0)
  =
U^*\widehat{\pi}(e,\0)U-\widehat{\pi}(e,\0)\in \I,
\end{align*}
since $\widehat{\pi}(e,\0)\in \I$.
\end{proof}

\begin{theorem}\label{c*-envelope left contractive}
\textup{\cite[Theorem 4.6]{KakKat11}} The $\ca$-envelope of the semicrossed products $\fA(\C,\ga,\text{contr})_l$ and $\fA(\C,\ga,\text{is})_l$ is the full corner $p\big(\B_\infty \rtimes_{\beta_\infty} \bbZ\big)p$ of the crossed product $\B_\infty \rtimes_{\beta_\infty} \bbZ$, where $p=\widehat{\pi}(e,\0)$.
\end{theorem}
\begin{proof}
We have shown that $\fD=p\big(\B_\infty \rtimes_{\beta_\infty} \bbZ\big)p$ is a $\ca$-cover of the semicrossed products. Let $\J$ be the \v{S}ilov ideal. First of all, note that $\J$ cannot intersect $\C$. Indeed if there was a $\widehat\pi(c,\0) \in \J \cap \C$ then 
\begin{align*}
\|c\| = \|\widehat{\pi}(c,\0)\|=\|\widehat{\pi}(c,\0) + \sca{\widehat{\pi}(c,\0)}\|=0.
\end{align*}
Since, $\J$ is $\beta_z$-invariant, there is an $x$ in the fixed point algebra and in $\J$. Suppose that $x$ has the form
\begin{align*}
x = \sum_{k=0}^n (Up)^k \widehat\pi(c_k,\0) (pU^*)^k,
\end{align*}
and without loss of generality we may assume that $k$ is the least such integer. Indeed, otherwise substitute $x$ with $x - UpU^*x$. Then for $a \in \ker\ga^\perp$ we have that
\begin{align*}
\J \ni xa = \sum_{k=0}^n (Up)^k \widehat\pi(c_k,\0) \widehat\pi(\ga_k(a),\0) (pU^*)^k = \widehat\pi(c_0a, \0).
\end{align*}
Therefore $\widehat\pi(c_0a) = 0$ which implies that $c_0 \in \ker\ga^\perp$. Now recall that $U$ is a unitary, hence
\begin{align*}
\widehat{\pi}(c_0, \0) 
& = \widehat\pi(c_0,\0)UU^* 
 = U \widehat\pi\left(\beta(c_0,\0)\right) U^* \\
& = U \widehat\pi(\ga(c_0),\theta(c_0),\0) U^* 
 = (Up) \widehat\pi(\ga(c_0),\0) (pU)^*.
\end{align*}
Consequently,
\begin{align*}
UpU^* x = UpU^* \left( \sum_{k=1}^n (Up)^k \widehat\pi(c_{k}',\0) (pU^*)^k \right)
= \sum_{k=1}^n (Up)^k \widehat\pi(c_{k}',\0) (pU^*)^k = x,
\end{align*}
and also $UpU^* x UpU^* = UpU^* x = x$. Thus
\begin{align*}
pU^* x Up = \sum_{k=0}^{n-1} (Up)^k \widehat\pi(c_{k+1}',\0) (pU^*)^k \in \J.
\end{align*}
By assumption $pU^* x Up = 0$, hence $x= UpU^*x = Up(pU^* x Up)pU^* = 0$.
\end{proof}

\section{An overview}\label{overview}

In this section we gather some remarks concerning the semicrossed products we have defined. We present them just for the semicrossed products that satisfy the left covariance relation. Of course one can get the analogues for the right case. However, one must distinguish the left from the right case, as we have already argued in Remark \ref{r:pet}.

Propositions \ref{left isomorphic 1} and \ref{left isomorphic 2} imply that $\fA(\C,\ga,\text{contr})_l \simeq \fA(\C,\ga,\text{is})_l$ and $\fA(\C,\ga,\text{co-is})_l\simeq \fA(\C,\ga,\text{un})_l$, and by the universal property of $\fA(\C,\ga,\text{contr})_l$, the identity map $\ell^1(\bbZ_+,\C,\ga)_l \rightarrow \ell^1(\bbZ_+,\C,\ga)_l$ extends to a unital completely contractive homomorphism of $\fA(\C,\ga,\text{contr})_l$ onto $\fA(\C,\ga,\text{co-is})_l$ (since every covariant co-isometric pair is a covariant contractive pair).

But there are cases where the semicrossed products are not completely isometrically isomorphic. Indeed, consider the dynamical system $(\C,\ga)$ of Example \ref{trivial}. Then the semicrossed products $\fA(\C,\ga,\text{co-is})_l$ and $\fA(\C,\ga,\text{un})_l$ are completely isometrically isomorphic to the disc algebra $\bbA(\bbD)$. On the other hand, by Theorem \ref{left only one}, the semicrossed products $\fA(\C,\ga,\text{contr})_l$ and $\fA(\C,\ga,\text{is})_l$ contain a copy of $\C$. Since $\bbA(\bbD)$ does not contain a copy of $\oC(\bbR_+ \cup \{\infty\})$ (the only $\ca$-algebra that lives in $\bbA(\bbD)$ is $\bbC$), we have that $\fA(\C,\ga,\text{is})_l$ and $\fA(\C,\ga,\text{co-is})_l$ cannot be completely isometrically isomorphic. The following proposition shows that this happens because $\ga$ is not injective.

\begin{proposition}
There is a unital completely isometric isomorphism $\Phi: \fA(\C,\ga,\text{contr})_l \rightarrow
\fA(\C,\ga,\text{co-is})_l$, such that $\Phi$ fixes $\C$ pointwise
if, and only if, $\ga$ is injective. The same holds for the right
case.
\end{proposition}
\begin{proof}
Let $\Phi\colon \fA(\C,\ga,\text{contr})_l \rightarrow \fA(\C,\ga,\text{co-is})_l$ be such that $\Phi(\gd_0 \otimes c)=\gd_0 \otimes c$. Then for $c\in \R_\ga$,
\begin{align*}
0
& =
\nor{\gd_0 \otimes c}_{\fA(\C,\ga,\text{co-is})_l}
  =
\nor{\Phi(\gd_0 \otimes c)}_{\fA(\C,\ga,\text{co-is})_l}\\
& =
\nor{\gd_0 \otimes c}_{\fA(\C,\ga,\text{contr})_l}
  =
\nor{c}_\C.
\end{align*}
Hence, $\R_\ga=(0)$, so $\ga$ is injective.

For the converse, assume that $\ga\colon \C \rightarrow \C$ is injective. Then the ideal $\R_\ga$ is trivial and the $\ca$-algebra $\B$ of Section \ref{left} is exactly $\C$. Hence, Theorems \ref{C*-envelope left co-isometric} and \ref{left only one} give that the map $\gd_n \otimes c \mapsto \gd_n \otimes c \in \C_\infty \rtimes_{\ga_\infty} \bbZ$ extends to a completely isometric homomorphism from $\fA(\C,\ga,\text{contr})_l$ onto $\fA(\C,\ga,\text{co-is})_l$.
\end{proof}

\begin{remark}\label{remark all the same}
\textup{By the previous proposition the mapping $\gd_n \otimes c \mapsto \gd_n \otimes c$ defines a completely isometric isomorphism between all pairs of semicrossed products if, and only, if $\ga$ is injective. In this case, they all share the same $\ca$-envelope, the crossed product $\C_\infty \rtimes_{\ga_\infty} \bbZ$.}

\textup{Finally, we note that when $\ga\colon \C \rightarrow \C$ is a $*$-automorphism, then $\C_\infty=\C$ and $\ga_\infty=\ga$, thus the $\ca$-envelope of the semicrossed products is the crossed product $\C \rtimes_\ga \bbZ$, and we recover \cite[Theorem 1.5]{Kak09}.}
\end{remark}

\begin{remark}\label{Peters}
\textup{Theorem \ref{theorem C^*-envelope right isometric} and Theorem \ref{c*-envelope left contractive} extend \cite[Theorem 4]{Petarx}. In \cite{Petarx} Peters studies the semicrossed products $\fA(\C,\ga,\text{is})_l$ and $\fA(\C,\ga,\text{is})_r$, where $\C$ is a commutative $\ca$-algebra $\oC(X)$ and $\ga(f)=f\circ \phi$, with $\phi\colon X \rightarrow X$ surjective, i.e. $\ga$ is injective. In this case, one can construct an extension $(\tX,\tphi)$ of $(X,\phi)$, where $\tphi$ is an homeomorphism. In a few words, let $\prod_{n\in \bbZ_+} X$ with its usual topology and define the subset
\begin{align*}
\tX = \{(x_n)\in \prod_{n\in \bbZ_+} X: \phi(x_{n+1})=x_n\},
\end{align*}
which is a compact Hausdorff space. Then the continuous map
\begin{align*}
\tphi\colon \tX \rightarrow \tX: (x_n) \mapsto \tphi((x_n))=(\phi(x_1),(x_n))
\end{align*}
is a homeomorphism of $\tX$ and by \cite[Theorem 4]{Petarx} the $\ca$-envelope of $\fA(\C,\ga,\text{is})_l$ and $\fA(\C,\ga,\text{is})_r$ is $\oC(\tX)\rtimes_{\tphi} \bbZ$. But by \cite[Corollary 4]{Petarx} $(\tX,\tphi)$ is exactly $(\oC(X)_\infty, \ga_\infty)$.}
\end{remark}

\section{Minimality}\label{minimality}

In this section, we indicate why we believe that the $\ca$-envelope of a semicrossed product is a good candidate for the $\ca$-algebra generated by a unital dynamical system. The reason is that properties of $(\C,\ga)$ pass naturally to the dynamical systems $(\B,\beta)$ and $(\B_\infty,\beta_\infty)$ and vice versa. To give an example, we prove a result that connects a notion of \emph{minimality} to a notion of \emph{simplicity}. There are a number of similar results in the literature. For example, Davidson and Roydor have proved that, for commutative multivariable dynamical systems, minimality is equivalent to the simplicity of the $\ca$-envelope of the tensor product \cite[Definition 5.1, Proposition 5.4]{DavR}. Also, there are well known criteria that give equivalence of simplicity of  a crossed product to properties of a dynamical system by a $*$-automorphism (for example see \cite{ArchSpiel93, KawTom90, Rea91} and/or \cite{Sie09}). However, none of these results fits in our case. For example \cite[Proposition 5.4]{DavR} is proved for dynamical systems where at least two $*$-endomorphisms participate. Moreover, in our context we deal in general with non-automorphic dynamical systems. Nevertheless, in most of the cases the $\ca$-envelope of a semicrossed product is a crossed product of a larger dynamical system and our intention is to show how results for crossed products can be used for semicrossed products with a small effort. Let us fix the notion of minimality that will be used throughout this section.

\begin{definition}\label{defn minimal}
\textup{A dynamical system $(\C,\ga)$ is called \emph{minimal} if there are no non-trivial ideals $\J$ of $\C$ that are $\ga$-invariant. Moreover, when $\ga$ is a $*$-automorphism, it is called \emph{bi-minimal} if there are no non-trivial \emph{$\ga$-bi-invariant} ideals $\J$ of $\C$, i.e. that satisfy $\ga(\J)=\J$.}
\end{definition}

\begin{remark}\label{unital for minimality}
\textup{\emph{When $\ga$ is a $*$-automorphism, then $(\C,\ga)$ is minimal if and only if it is bi-minimal}. Indeed, let $\J$ such that $\ga(\J)\subseteq \J$, and define $\I= \overline{\cup_n \ga^{-n}(\J)}$. Then $\I$ is a $\ga$-bi-invariant ideal, hence if $(\C,\ga)$ is bi-minimal, then $e_\C \in \I$. Thus, for $\eps>0$, there is $n_0$ and $x\in \J$ such that $\nor{e_\C - \ga^{-n_0}(x)} <\eps$. Since $\ga$ is isometric and unital we get
\begin{align*}
 \nor{e_\C - x}=\nor{\ga^{-n_0}(e_\C-x)}= \nor{e_\C - \ga^{-n_0}(x)} <\eps.
\end{align*}
Since $\eps>0$ was arbitrary and $x\in \J$ we get that $e_\C \in \J$, hence $\J=\C$. Thus $(\C,\ga)$ is minimal. The converse is trivial.}
\end{remark}

Before we proceed to the first main Theorem of this section, let us briefly discuss the Fourier transform of a crossed product. Let $(\C,\ga)$ be a dynamical system such that $\ga\colon \C \rightarrow \C$ is a $*$-automorphism. If $E\colon \C\rtimes_\ga \bbZ \rightarrow \C\rtimes_\ga \bbZ^\beta\equiv \C$ is the conditional expectation of the crossed product, we define the Fourier co-efficients
\begin{align*}
E_n\colon \C\rtimes_\ga \bbZ \rightarrow \C\rtimes_\ga \bbZ^\beta \equiv \C \colon F \mapsto E_n(F):=E(U^{-n}F), n\in \bbZ.
\end{align*}
A F\'{e}jer-type Lemma shows that the C\'{e}saro means of the Fourier monomials $U^n E_n(F)$ converge to $F$ in norm. Moreover, if $\I$ is a non-zero ideal of $\C\rtimes_\ga \bbZ$, then $E_n(\I)$ is a non-zero $\ga$-bi-invariant ideal in $\C$.

\begin{definition}\label{Fourier defn}
\textup{Let $(\C,\ga)$ be a dynamical system such that $\ga\colon \C \rightarrow \C$ is a $*$-automorphism. An ideal $\I$ of $\C\rtimes_\ga \bbZ$ is called \emph{Fourier-invariant} if $E_n(\I)\subseteq \I$ for all $n\in \bbZ$.}
\end{definition}

Note that this is equivalent to saying that the Fourier monomials $U^nE_n(F)$ are in $\I$, for every $F\in \I$ and $n\in \bbZ$.

For the next Theorem we use the constructions of Section \ref{left}.

\begin{theorem}\label{minimal theorem}
The following are equivalent:
\begin{enumerate}
 \item $(\C,\ga)$ is minimal,
 \item $(\B,\beta)$ is minimal,
 \item $(\B_\infty,\beta_{\infty})$ is bi-minimal,
 \item the crossed product $\B_\infty\rtimes_{\beta_\infty} \bbZ$
 has no non-trivial Fourier-invariant ideals.
\end{enumerate}
If any of the previous conditions holds, then $\ga$ is injective and
$\C=\B$. Moreover, the semicrossed products we have defined with
respect to the collections $\F_{t,l}$, for $t=1,2,3,4$, are
completely isometrically isomorphic to one another and share the
same $\ca$-envelope $\C_\infty\rtimes_{\ga_\infty} \bbZ$.
\end{theorem}
\begin{proof}
If $(\C,\ga)$ is minimal, then $\ga$ is injective, otherwise $\ker\ga$ would be an $\ga$-invariant ideal. Hence, $T=(0)$ which gives that $(\B,\beta)=(\C,\ga)$, thus $(\B,\beta)$ is minimal.

Conversely, assume that $(\B,\beta)$ is minimal. Then $T=(0)$, since $T$ is $\beta$-invariant and cannot be equal to $\B$. Thus $(\B,\beta)=(\C,\ga)$, hence $(\C,\ga)$ is minimal. So $\ga$ is injective, otherwise $\M(\ker\ga)\neq (0)$ which leads to the contradiction $T\neq (0)$.

Assume that $(\C,\ga)$ is minimal (hence $\ga$ is injective), and let $\J\neq (0)$ be a $\ga_\infty$-invariant ideal in $\C_\infty$. Then $\J$ has non-trivial intersection with $\C$. Indeed, assume that $\J \cap \C =(0)$ and let $0\neq c\in \J \cap \C_{n_0}$ (if there is not such an $n_0$ then $\J=(0)$). Then
\begin{align*}
\ga_\infty^{n_0}(c) \in \ga_\infty^{n_0}(\J) \cap \ga_\infty^{n_0}(\C_{n_0})  \subseteq \J \cap \C = (0).
\end{align*}
Hence, $c\in \ker\ga_\infty^{n_0}$, which leads to $\ker\ga_\infty \neq (0)$, a contradiction. Now, let $I= \J \cap \C \neq (0)$ which is an ideal in $\C$; then
\begin{align*}
\ga(I)
=
\ga_\infty(I)
\subseteq
\ga_\infty(\J) \cap \ga_\infty(\C)
\subseteq
\J \cap \C
=
I,
\end{align*}
hence, $I$ is a non-zero $\ga$-invariant ideal of $\C$. Thus $I=\C$, therefore $e_{\C_\infty}=e_\C \in I \subseteq \J$. So $\J=\C_\infty$.

To end the proof, assume that $(\B_\infty,\beta_\infty)$ is minimal, and let $J\neq(0)$ be a $\beta$-invariant ideal in $\B$. Let $\J_\infty$ be the $\ca$-subalgebra of $\B_\infty$ defined by the directed system
\begin{align*}
J \stackrel{\beta}{\longrightarrow} J \stackrel{\beta}{\longrightarrow} J \stackrel{\beta}{\longrightarrow} \dots.
\end{align*}
It is well defined because $\beta$ restricts to an injective $*$-endomorphism of $J$. It is easy to see that $\J_\infty$ is a non zero ideal in $\B_\infty$ and, moreover, that is $\beta_\infty$-invariant. Hence $\J_\infty=\B_\infty$. Then $J=\J_\infty \cap \B=\B$.

Finally the equivalence $[(3) \Leftrightarrow (4)]$ is a standard result for crossed products.
\end{proof}

\begin{remark}\label{simplicity remark}
\textup{It is immediate that simplicity of the crossed product implies minimality of the dynamical system. But there is no hope of proving the converse in general. As Davidson and Roydor point out in \cite[Remark 5.12]{DavR} if we consider the minimal dynamical system $(\C,\ga)$ where $\C=\oC(\{x\})$ and $\ga=\id$, then $\B_\infty\rtimes_{\beta_\infty} \bbZ\simeq \C \rtimes_{\id} \bbZ \simeq \oC(\bbT)$ which is not simple. To go even further, assume that we are given the dynamical system $(\C,\id)$, where $\C$ is simple (thus the dynamical system is bi-minimal). Then $\C \rtimes_{\id} \bbZ$ is isomorphic to $\C \widehat{\otimes} \oC(\bbT)$. Hence, given a non-trivial ideal $J \vartriangleleft \oC(\bbT)$ we have that $\C \otimes J$ is a non-trivial ideal in the crossed product.}

\textup{On the other hand, there are well known dynamical systems with $*$-automorphisms, such that the crossed product they produce is a simple $\ca$-algebra. An example is $\A_\theta$ with $\theta \in \bbR \setminus \bbQ$ (see \cite[Theorem VIII.3.9]{Dav96}).}
\end{remark}

Let us drop to the case of commutative dynamical systems. So, let $X$ be a compact Hausdorff space and $\phi\colon X \rightarrow X$ a continuous map. For simplicity we write $(X,\phi)$ instead of $(\oC(X),\ga)$, where $\ga(f):=f\circ \phi$.

Assume that $\phi$ is not surjective, i.e. $\ga$ is not injective, then we can add the tail we produced in Section \ref{left}. We use notation as in \cite{DavR}. Define $U= X \setminus \phi(X)$, $T=\{(u,k):u \in \overline{U}, k<0\}$ and $X^T=X \sqcup T$. Then the continuous mapping $\phi^T\colon X^T \rightarrow X^T$ with $\phi^T|_X=\phi$ and
\begin{align*}
\phi^T(u,k)=(u,k+1), \text{ for } k<-1, \text{ and } \phi^T(u,-1)=u
\end{align*}
is surjective. Moreover, $(X^T,\phi^T)$ is the dynamical system $(\B,\beta)$ we construct in Section \ref{left} for $(\oC(X),\ga)$.

When $\phi$ is surjective, we can use the projective limit described in Remark \ref{Peters} and get the dynamical system $(\tX,\tphi)$.

\begin{theorem}\label{simplicity}
Let $(X,\phi)$ be a dynamical system, where $X$ is a compact Haudorff space. Minimality implies surjectivity and the following are equivalent:
\begin{enumerate}
\item $(X,\phi)$ is minimal and $X$ is infinite,
\item $(\tX,\tphi)$ is minimal and $\tX$ is infinite,
\item the crossed product $\oC(\tX)\rtimes_{\tphi} \bbZ$ is simple.
\end{enumerate}
If any of the previous conditions hold, the semicrossed products we have defined with respect to the collections $\F_{t,l}$, for $t=1,2,3,4$, are completely isometrically isomorphic to one another and share the same $\ca$-envelope $\oC(\tX)\rtimes_{\tphi} \bbZ$.
\end{theorem}
\begin{proof}
Note that $X$ is infinite if and only if $\tX$ is infinite. So, Theorem \ref{minimal theorem} gives the equivalence $[(1)\Leftrightarrow (2)]$. Also $[(2)\Rightarrow (3)]$ is \cite[Theorem VIII.3.9]{Dav96}. To finish the proof, assume that $\oC(\tX)\rtimes_{\tphi} \bbZ$ is simple. Then by \cite[Corollary]{ArchSpiel93} $(\tX,\tphi)$ is minimal and topologically free, i.e. for every $0\neq k \in \bbZ$ the open set $\{x\in \tX: \tphi^k(x)\neq x\}$ is dense in $\tX$ (see \cite[Definition 1]{ArchSpiel93} and the remarks following it). If $\tX$ were finite, then by the pigeonhole principle there is at least one periodic point $x_0\in \tX$, say with period $m$. Then $x_0 \notin \{x\in \tX: \tphi^m(x)\neq x\}$, and since $\{x\in \tX: \tphi^m(x)\neq x\}$ is dense in $\tX$, thus nonempty, we can assume that $\{x\in \tX: \tphi^m(x)\neq x\}= \{x_1,\dots,x_l\} $. Since $X$ is Hausdorff there is an open neighborhood $U$ of $x_0$ such that none of $x_1, \dots, x_l$ is in $U$. Hence, $\{x\in \tX: \tphi^m(x)\neq x\}$ is not dense in $\tX$, which is a contradiction. Thus $\tX$ is infinite and the proof is complete.
\end{proof}

\begin{remark}\label{semisimple}
\textup{Simplicity of the $\ca$-envelope induces semi-simplicity of the semicrossed product (see \cite[Proposition 3]{Petarx}), but the converse is false. For example, assume $X=\bbT$ and $\phi\colon \bbT \rightarrow \bbT$ be rotation by a rational angle $\theta$. It is obvious that $\phi$ is surjective, thus the semicrossed products are completely isometrically isomorphic. Every point in $\bbT$ is recurrent, hence by \cite[Theorem 10]{DonsKatMan01} the semicrossed products are semisimple. But $\A_\theta$ is not simple.}
\end{remark}

Theorem \ref{simplicity} provides the identification of all ideals of any $\ca$-cover of the semicrossed products as boundary ideals.

\begin{corollary}
Let $(X,\phi)$ be a dynamical system, where $X$ is a compact Haudorff space. If any of $(1)-(3)$ of Theorem \ref{simplicity} holds, then the semicrossed products are u.c.is.is, and if $\I$ is an ideal in a $\ca$-cover $\fC$ of the semicrossed product, then it is boundary.
\end{corollary}
\begin{proof}
Let $\I \lhd \fC$ and the $*$-epimorphism $\Phi\colon \fC \rightarrow \oC(\tX)\rtimes_{\tphi} \bbZ$. Since the $\ca$-envelope is simple, then $\Phi(\I)=(0)$, hence $\I \subseteq \ker\Phi$; thus it is a boundary ideal for the semicrossed product, since $\ker\Phi$ is the \v{S}ilov ideal, i.e. the biggest boundary ideal.
\end{proof}

\begin{question}
\textup{Note that for the proof of Theorem \ref{simplicity} we have used \cite[Corollary]{ArchSpiel93}, which holds for commutative dynamical systems. For non- commutative and automorphic dynamical systems there exist criterias that lead to simplicity of the crossed product. Hence, the natural question that is raised here is if there is an analogue of Theorem \ref{simplicity} for non-commutative dynamical systems $(\C,\ga)$ at least when the spectrum $\widehat{\C}$ is Hausdorff. The first step to that direction would be to answer the following question:}

\begin{quote}
\noindent Let $(\C,\ga)$ be a unital dynamical system such that the spectrum $\widehat{\C}$ is a Hausdorff space. Then $(\C,\ga)$ is minimal and topologically free if, and only if, $(\C_\infty,\ga_\infty)$ is minimal and topologically free.
\end{quote}

\textup{Recall that an arbitrary dynamical system $(\C,\ga)$ is called \emph{topological free} if for any $n_1,\dots,n_k \in \bbZ_+\setminus\{0\}$, $\cap_{i=1}^k \{x\in \widehat{\C}: x\circ \ga^{n_i} \neq x\}$ is dense in $\widehat{\C}$. When $\widehat{\C}$ is Hausdorff this is equivalent to saying that $\{x\in \widehat{\C}: x\circ \ga^{n} = x\}$ has empty interior for any $n \geq 1$ (by the remarks following \cite[Proposition 1]{ArchSpiel93}). In the case where $\ga$ is a $*$-automorphism, this is equivalent to the usual topological freeness \cite[Definition 1]{ArchSpiel93}, since $\{x\in \widehat{\C}: x\circ \ga^{n} \neq x\}= \{x\in \widehat{\C}: x\circ \ga^{-n} \neq x\}$ for any $n\geq 1$.}
\end{question}

\begin{acknow}
\textup{I wish to give my sincere thanks to E. Katsoulis for his helpful remarks and advices. I also wish to thank A. Katavolos for his kind help and advice during the preparation of this paper. Finally, I wish to thank Judy Vs. for the support and inspiration.}
\end{acknow}


\end{document}